\documentclass[a4paper, 12pt]{article}

\usepackage[sort&compress]{natbib}
\bibpunct{(}{)}{;}{a}{}{,} 

\usepackage{amsthm, amsmath, amssymb, mathrsfs, multirow, url, subfigure}
\usepackage{graphicx} 
\usepackage{ifthen} 
\usepackage{amsfonts}
\usepackage[usenames]{color}
\usepackage{fullpage}
\usepackage{tikz}

\theoremstyle{plain} 
\newtheorem{thm}{Theorem}

\newtheorem*{cor0}{Corollary}
\newtheorem*{thm.rough}{``Theorem''}

\theoremstyle{definition}
\newtheorem{defn}{Definition}

\theoremstyle{remark} 

\newtheorem*{astep}{A-step}
\newtheorem*{pstep}{P-step}
\newtheorem*{cstep}{C-step}

\newcommand{\prob}{\mathsf{P}}

\newcommand{\mpl}{\mathsf{mpl}}

\newcommand{\bin}{{\sf Bin}}
\newcommand{\unif}{{\sf Unif}}
\newcommand{\nm}{{\sf N}}

\newcommand{\chisq}{{\sf ChiSq}}

\newcommand{\A}{\mathcal{A}} 
 
\newcommand{\RR}{\mathbb{R}}
 
\newcommand{\U}{\mathscr{U}}

\newcommand{\XX}{\mathbb{X}}

\newcommand{\UU}{\mathbb{U}}

\newcommand{\FF}{\mathbb{F}}

\newcommand{\xbar}{\bar{x}}

\newcommand{\ubar}{\bar{u}}

\newcommand{\G}{\mathscr{G}}
\newcommand{\Gbar}{\overline{\mathscr{G}}}
\newcommand{\Gtilde}{\widetilde{\mathscr{G}}}

\renewcommand{\S}{\mathcal{S}}
\renewcommand{\SS}{\mathbb{S}}

\newcommand{\model}{\mathscr{P}}
\newcommand{\uPi}{\overline{\Pi}}
\newcommand{\lPi}{\underline{\Pi}}
\newcommand{\uPif}{\widetilde\Pi}

\newcommand{\pif}{\pi^{(\text{\sc f})}}

\title{An imprecise-probabilistic characterization of frequentist statistical inference\footnote{On December 21st, 2020, Professor D.~A.~S.~Fraser, from the University of Toronto, passed away at the age of 95.  His work had and will continue to have a tremendous impact on the field of statistics and on me personally.  In fact, Don's ideas were influential to the present developments, so it's an honor for me to dedicate this paper to his memory.}}
\author{Ryan Martin\footnote{Department of Statistics, North Carolina State University, {\tt rgmarti3@ncsu.edu}}}
\date{\today}

\begin{document}

\maketitle 

\begin{abstract}  
Between the two dominant schools of thought in statistics, namely, Bayesian and classical/frequentist, a main difference is that the former is grounded in the mathematically rigorous theory of probability while the latter is not.  In this paper, I show that the latter is grounded in a different but equally mathematically rigorous theory of imprecise probability.  Specifically, I show that for every suitable testing or confidence procedure with error rate control guarantees, there exists a consonant plausibility function whose derived testing or confidence procedure is no less efficient. Beyond its foundational implications, this characterization has at least two important practical consequences: first, it simplifies the interpretation of p-values and confidence regions, thus creating opportunities for improved education and scientific communication; second, the constructive proof of the main results leads to a strategy for new and improved methods in challenging inference problems.  

\smallskip

\emph{Keywords and phrases:} confidence region; consonance; imprecise probability; inferential model; p-value; possibility measure; random set; validity.
\end{abstract}


\section{Introduction}
\label{S:intro} 

In comparisons of the Bayesian and classical/frequentist approaches to statistical inference, an often-cited difference is that the former is rooted in the mathematically rigorous theory of probability while the latter is not.  That is, the Bayesian approach proceeds by specifying prior beliefs about unknowns and updating those, in light of the observed data, using the standard rules of probability, whereas a frequentist approach is less rigid, some might even say ``ad hoc.''  Since being rooted in a mathematically rigorous framework is generally better than not, 
the connection to probability theory is often interpreted as a win for the Bayesian side.  However, that interpretation is based on the incorrect assumption that probability is the only mathematically rigorous model for uncertainty quantification.  The goal of the present paper is to demonstrate that the so-called ``frequentist approach'' is rooted in the mathematically rigorous theory of {\em imprecise probability}, more specifically, that based on distributions of nested random sets, i.e., {\em consonant plausibility functions} \citep[e.g.,][]{shafer1976} or {\em possibility measures} \citep[e.g.,][]{dubois.prade.book}.  As I explain below, this fundamental characterization is impactful in various ways.  

One benefit of this new characterization is in the communication of statistical concepts and the interpretation of statistical results.  As many of us know first hand, as instructors of introductory statistics courses, explaining how to interpret classical/frequentist results can be a challenge.  Indeed, we stress to our students that 
\begin{itemize}
\item a p-value {\em doesn't} represent the probability that the null hypothesis is true, and 
\vspace{-2mm}
\item ``95\% confidence'' {\em doesn't} mean that the true parameter is in the stated interval with probability 0.95.
\end{itemize}
But what {\em does} a p-value or ``95\% confidence'' actually mean?  When we attempt to explain ``95\% confidence'' in terms of the proportion of possible intervals that could have been observed, students' eyes glaze over.  The reason this explanation fails is that the fixed-data interpretation of p-values and confidence intervals should be different from the across-data statistical properties they satisfy.  
Fortunately, the intuitive or pragmatic interpretations typically offered to students are mathematically correct as well.  I tell my students that, for the given data, posited model, etc., 
\begin{itemize}
\item a p-value is a measure of how {\em plausible} the null hypothesis is, and
\vspace{-2mm}
\item a confidence interval is a set of sufficiently {\em plausible} parameter values.  
\end{itemize}
This is more than just replacing the word ``probable'' with ``plausible'' because, as I show below, the p-values and confidence intervals are features derived from a data-dependent plausibility measure, a well-defined mathematical object.  This is similar to how Bayesian probabilities are summaries of the data-dependent posterior distribution, but with a critical difference: in general, there is no Bayesian posterior probability distribution whose summaries agree with the p-values and confidence intervals, both in terms of their numerical values for fixed samples and their sampling distribution properties across samples.  By taking imprecise probabilities into consideration, the pragmatic interpretation of p-values and confidence intervals described above can be made mathematically rigorous.  

The proposed shift in perspective from precise to imprecise probability would also be beneficial beyond the classroom.  One statistical factor contributing to the replication crisis currently plaguing science is a misunderstanding of statistical reasoning, in particular, the meaning of ``statistical significance''  \citep[e.g.,][]{wasserstein.lazar.asa, wasserstein.schirm.lazar.2019}.  Pushes to ban p-values \citep{pvalue.ban} or to abandon statistical significance \citep{mcshane.etal.rss} are aimed at preventing ``$p < 0.05$'' from being misinterpreted as direct evidence supporting a scientific discovery claim.  But nothing in the original literature supports such an interpretation.  In the 1925 edition of \citet{fisher1973b}, and even further back to Edgeworth in 1885, comparing p-values to a pre-determined threshold was intended ``simply as a tool to indicate when a result warrants further scrutiny'' \citep[][p.~2]{wasserstein.schirm.lazar.2019}.  Fisher definitely did not intend for statistical significance to imply scientific discovery, but apparently he did see p-values as a measure of how plausible the null hypothesis is based on the data, posited model, etc.  Unfortunately, p-values look like probabilities and, without a more appropriate mathematical framework to explain them in, the community defaulted to the familiar probability theory and interpreted small p-values as an indication that the alternative is probable, hence a discovery.  What's changed from Fisher's time is that now there exists an imprecise probability framework wherein the mathematics can be developed in a way that aligns with the p-value's intended interpretation.  For example, the sub-additivity property of plausibility measures leads to the following:
\begin{quote}
If a hypothesis is implausible, then its complement must be plausible; however, high plausibility alone isn't evidence of the complementary hypothesis's truthfulness, since the plausibility measure's dual may give it low support. 
\end{quote}
Therefore, the basic mathematical properties of plausibility refute misinterpretations like ``a small p-value (alone) implies a scientific discovery.'' Armed with this understanding, the statistical community can correct its communication errors, no bans required.
 
A second statistical factor relevant to concerns about replicability is the use of methods that are reliable in the sense that they can provably avoid making systematic errors.  This is important because statistical methods are being used as part of the scientific decision-making process, e.g., what existing theories are rejected in favor of new ones.  This notion of reliability is immediately seen as related to controlling error rates in a frequentist sense.  As I'll demonstrate below, the characterization in terms of imprecise probability is important---even essential---for this practical aspect as well.

My starting point is to consider a very general kind of inferential output, namely, a data-dependent capacity; see Definition~\ref{def:im}.  The idea is that the data-dependent capacity, evaluated at any assertion about the unknowns, would measure the plausibility of that assertion, relative to the data and posited model.  Capacities have minimal restriction on their mathematical form, but the statistical context in which they're to be used will introduce more structure.  Indeed, since assigning small plausibility values to assertions that happen to be true may lead to erroneous conclusions, it would be desirable if the capacity controlled the frequency of such errors in a specific sense.  I refer to this error rate control property as {\em validity}; see Definition~\ref{def:valid}.  Not surprisingly, requiring the capacity to satisfy validity tightens up much of its original flexibility.  What mathematical structure can be imposed on the capacity that's consistent with validity?  

Towards an answer to this question, I start with a basic procedure, e.g., a test or confidence region, already known to possess a form of frequentist error rate control, and attempt to construct a suitable capacity that achieves the validity property.  In particular, after some setup and notation in Section~\ref{S:ims}, in Section~\ref{SS:isnt.valid}, I consider the conversion of a confidence region into so-called {\em confidence distribution}, a data-dependent probability measure on the parameter space.  There I show, in the context of some specific examples, that additivity is too demanding of a constraint on the capacity, i.e., additivity/precision and validity are individually too strong to be compatible \citep{balch.martin.ferson.2017}.  Still working with a given confidence region, I show in Section~\ref{SS:is.valid} how to construct an imprecise probability that achieves validity for all possible assertions.  The resulting capacity has the mathematical structure of a consonant plausibility function or, equivalently, a possibility measure.  \citet[][Ch.~10]{shafer1976} argues that consonance is a natural structure to impose in statistical inference problems, and there is literature \citep[e.g.,][Ch.~4.6.1]{destercke.dubois.2014} that highlights the connection between consonance and statistical calibration properties; see, also, \citet{balch2012}, \citet{denoeux.li.2018}, and \citet{imposs}.  However, to my knowledge, the particular reverse-engineering of a consonant plausibility function from a given confidence region proposed here is new.  

One can construct consonant plausibility functions that satisfy the validity property directly, without having a given confidence region to start with.  That's what the developments in \citet{imbasics, imcond, immarg, imbook} are about.  In particular, their framework starts with a representation of the statistical model in terms of an association between data, model parameters, and unobservable auxiliary variables, and their construction proceeds by introducing a random set targeting the unobserved value of that auxiliary variable corresponding to the observed data.  I review their theory and a certain generalization in Section~\ref{S:source}.  It turns out that an even stronger characterization of frequentist inference can be given, and that's the topic of Section~\ref{S:main}.   Roughly, Theorems~\ref{thm:complete}--\ref{thm:complete.test} in that section together can be summarized by the following 
\begin{thm.rough}
Given a testing or confidence procedure which, in addition to some mild regularity conditions, satisfies the frequentist error rate control property in \eqref{eq:size} and \eqref{eq:coverage}, respectively, there exists a consonant plausibility function---with the additional structure described in Section~\ref{S:source}---whose corresponding testing or confidence procedure is no less efficient than the one given.  
\end{thm.rough}

This has three important consequences.  First, it says that a shift in thinking about the ``frequentist approach'' in the traditional terms of testing and confidence procedures to valid imprecise probabilities costs nothing.  That is, every good frequentist procedure currently in use, along with any that could ever be developed, naturally fits in the proposed imprecise probability framework.  Second, it shows that the construction described in Section~\ref{S:source} is not just one way to construct imprecise probabilities whose derived procedures have good frequentist properties, it's effectively the only way.  Third, in addition to not costing anything, there are potential gains.  After some technical remarks and illustrations in Section~\ref{S:remarks}--\ref{S:examples}, I show how the constructive proof of the above ``theorem,'' along with the recent developments in \citet{wasserman.universal}, can be used to construct new procedures for practically challenging problems.  I illustrate this approach with two examples in Section~\ref{S:more}, namely, testing the number of components in mixture models and testing for monotonicity in an unknown density function.  To my knowledge, the test for monotonicity is the first in the literature that provably controls Type~I error.  

The results in this paper make some potentially unexpected connections between imprecise probability and classical/modern statistical theory.  So I take the opportunity in Section~\ref{S:open} to list out a few possible extensions and new directions to pursue.  Finally, some concluding remarks are given in Section~\ref{S:discuss}.

\section{Valid (imprecise) probabilistic inference}
\label{S:ims}

\subsection{Setup and notation}

To set the scene, suppose observable data $X \in \XX$ is modeled by a distribution $\prob_{X|\theta}$ indexed by a parameter $\theta \in \Theta$; then the statistical model is $\model = \{\prob_{X|\theta}: \theta \in \Theta\}$.  Note that the setup here is quite general: $X$, $\theta$, or both can be scalars, vectors, or something else.  A typical case is where $X=(X_1,\ldots,X_n)$ is a collection of $n$ independent and identically distributed (iid) random vectors of dimension $d \geq 1$ and $\theta$ is a vector taking values in $\Theta \subseteq \RR^q$, for $q \geq 1$.  But nothing in what follows requires that data points be iid, that the sample size $n$ be large, or that the model parameter $\theta$ be finite-dimensional.  In fact, I will consider examples below wherein $\theta$ represents the unknown distribution/density function, an infinite-dimensional quantity.  In some cases, the inferential target will be a lower-dimensional feature $\phi=\phi(\theta)$ of the full parameter, taking values in $\Phi = \phi(\Theta)$.  Throughout, when I'm referring to the true value of the parameter, I'll use the notation $\theta$ and $\phi$, but for generic values of these parameter, I'll use the variations $\vartheta$ and $\varphi$.  

The so-called ``frequentist approach'' is focused on procedures---hypothesis tests, confidence regions, etc.---motivated strictly by the statistical properties they satisfy.  For future reference, let me record here the definition of hypothesis tests and confidence regions and their relevant properties.
\begin{itemize}
\item Let $\Theta_0$ be a proper subset of $\Theta$ and consider the (possibly composite) null hypothesis $H_0: \theta \in \Theta_0$.  Let $T_\alpha: \XX \to \{0,1\}$ denote a collection of maps, indexed by $\alpha \in [0,1]$, that define a family of (non-randomized\footnote{I will not consider randomized tests here.  However, \citet{imconformal} considered randomization in a different context but from the same perspective I'm taking in this paper.  So I expect that randomized tests can be treated similarly.})
tests
\[ \text{reject $H_0$ based on data $x$ if and only if $T_\alpha(x)=1$}. \]
I'll say that these are {\em size-$\alpha$ tests} if they control Type~I error at level $\alpha$, i.e., if 
\begin{equation}
\label{eq:size}
\sup_{\theta \in \Theta_0} \prob_{X|\theta}\{T_\alpha(X) = 1\} \leq \alpha, \quad \text{for all $\alpha \in [0,1]$}. 
\end{equation}
\item Consider a family of set-valued maps $C_\alpha: \XX \to 2^\Phi$, indexed by $\alpha \in [0,1]$.  I'll say that these are {\em $100(1-\alpha)$\% confidence regions} for $\phi=\phi(\theta)$ if they have coverage probability at least $1-\alpha$, i.e., if
\begin{equation}
\label{eq:coverage}
\inf_{\theta \in \Theta} \prob_{X|\theta} \{C_\alpha(X) \ni \phi(\theta)\} \geq 1-\alpha, \quad \text{for all $\alpha \in [0,1]$}. 
\end{equation}
\end{itemize}
The focus on procedures makes the frequentist approach difficult to compare, at a fundamental level, with the Bayesian approach wherein the posterior distribution is the primitive and procedures are derivatives thereof.  To put the two frameworks on roughly the same playing field, and to satisfy statisticians' desire to quantify uncertainty about unknowns via (something like) probability, I propose the following formulation.  

Let $\A \subseteq 2^\Theta$ denote a collection of subsets of $\Theta$ which I'll interpret as the collection of assertions about $\theta$ for inferences can be drawn.  That is, I make no distinction between $A$ and the assertion ``$\theta \in A$.''  When dealing with classical probability measures, in the sense of Kolmogorov, if $\Theta$ is uncountable, then I'll implicitly assume that $\A$ is a proper $\sigma$-algebra of $2^\Theta$, but in other cases no such restriction will be needed.  If there is a feature $\phi=\phi(\theta)$ of particular interest, then $\A$ should contain those assertions $A$ of the form $A = \{\vartheta: \phi(\vartheta) \in B\}$, for all relevant assertions $B$ about $\phi$ satisfying the necessary measurability conditions, if any.  At the very least, $\A$ should contain both $\varnothing$ and $\Theta$ and be closed under complementation.  For the situations I have in mind, however, the collection $\A$ should be rich, e.g., at least the Borel $\sigma$-algebra.  The reason is that the goal is to construct something comparable to a Bayesian posterior distribution supported on $(\Theta, \A)$; but see Section~\ref{SS:rich} for more on this point.  

Next, following \citet{choquet1953}, a function $g: \A \to [0,1]$ is called a {\em capacity} if $g(\varnothing) = 0$, $g(\Theta) = 1$, and, for any $A,A' \in \A$, if $A \subseteq A'$, then $g(A) \leq g(A')$.  This is the most flexible---therefore, most complex---imprecise probability model \citep[e.g.,][]{destercke.dubois.2014}, and more structure will be added later when necessary.  I'll also assume that the capacity is {\em sub-additive} in the sense that 
\[ g(A \cup A') \leq g(A) + g(A'), \quad \text{for all disjoint $A,A' \in \A$}. \]
In addition to ordinary/precise probabilities, this setup includes all of the common imprecise probability models. 

\begin{defn}
\label{def:im}
An {\em inferential model} or IM is a mapping from $(x, \model, \ldots)$ to a data-dependent, sub-additive capacity $\uPi_x$ defined on $(\Theta,\A)$.  
\end{defn}

This definition covers cases where the IM output is a probability distribution, for example, Bayes or empirical Bayes \citep{ghosh-etal-book}, fiducial \citep{fisher1935a, fisher1973}, generalized fiducial \citep{hannig.review}, structural \citep{fraser1968}, or confidence distributions \citep{schweder.hjort.book, xie.singh.2012}.  In each case, the ``\ldots'' corresponds to the additional information required to carry out the IM construction, e.g., the Bayesian construction requires a prior distribution.  There's reason to consider IMs whose output is a genuine imprecise probability, so I'll discuss this less-familiar case further below.  

Define the dual/conjugate to the capacity $\uPi_x$ as 
\[ \lPi_x(A) = 1-\uPi_x(A^c), \quad A \in \A. \]
By the assumed sub-additivity of $\uPi_x$, it follows that  
\begin{equation}
\label{eq:order}
\lPi_x(A) \leq \uPi_x(A), \quad A \in \A. 
\end{equation}
This explains the lower- and upper-bar notation and justifies referring to $\lPi_x$ and $\uPi_x$ as lower and upper probabilities, respectively.  A behavioral interpretation of the lower and upper probabilities can be given, but I'll not do so here; see, e.g.,  \citet{walley1991} and \citet{lower.previsions.book} for a gambling-based perspective that leads to a generalization of de~Finetti's classical subjective interpretation of probability.  As an alternative to the behavioral or gambling interpretation, as some might prefer in the context of scientific/statistical inference, the data analyst can simply treat thee pair $(\lPi_x,\uPi_x)$ as determining data-dependent degrees of belief about $\theta$.  That is, small $\uPi_x(A)$ implies that the data $x$ doesn't support the truthfulness of the assertion $A$; similarly, large $\lPi_x(A)$ implies that data $x$ does support the truthfulness of $A$; otherwise, if $\lPi_x(A)$ is small and $\uPi_x(A)$ is large, then data $x$ is not sufficiently informative to make a definitive judgment about $A$.  

For the reader who is unfamiliar with genuinely non-additive capacities, I'll give a quick example.  Let $\mathcal{T}$ denote a {\em random set} \citep[e.g.,][]{molchanov2005, nguyen.book} that takes values in the power set of $\Theta$.  That random set has a probability distribution, which I'll denote by $\prob_{\mathcal{T}}$. Unlike a random variable, whose realizations can be, e.g., either $\leq 10$ or $> 10$, realizations of a random set compared to a fixed subset have three possibilities.  As illustrated in Figure~\ref{fig:rand.set}, $\mathcal{T}$ could be (a)~contained in $A$, (b)~contained in $A^c$, or (c)~have non-empty intersection with both $A$ and $A^c$.  It's having three possible outcomes that leads to non-additivity.  Indeed, define a capacity $\lPi$ on $\Theta$ as 
\[ \uPi(A) = \prob_{\mathcal{T}}(\mathcal{T} \cap A \neq \varnothing), \quad A \subseteq \Theta, \]
and its dual $\lPi(A) = 1 - \uPi(A^c) = \prob_{\mathcal{T}}(\mathcal{T} \subseteq A)$.  Then it's clear that these are sub- and super-additive, respectively, and, in particular, that \eqref{eq:order} holds for all $A$.  

\begin{figure}
\begin{center}
\subfigure[$\mathcal{T} \subseteq A$]{
\begin{tikzpicture}
\filldraw[color=black!20, fill=black!5, rotate=50] (0.5,0) ellipse (2.1 and 1.1);
\filldraw[black] (1.1, 1.8) circle (0pt) node{$A$};
\filldraw[color=red!20, fill=red!5, thick](0,0) circle (0.7);
\filldraw[black] (0,0) circle (0pt) node{$\mathcal{T}$};
\end{tikzpicture}
} \hspace{5mm} %
\subfigure[$\mathcal{T} \subseteq A^c$]{
\begin{tikzpicture}
\filldraw[color=black!20, fill=black!5, rotate=50] (0.5,0) ellipse (2.1 and 1.1);
\filldraw[black] (1.1, 1.8) circle (0pt) node{$A$};
\filldraw[color=red!20, fill=red!5, thick](2.1, -0.8) circle (0.7);
\filldraw[black] (2.1, -0.8) circle (0pt) node{$\mathcal{T}$};
\end{tikzpicture}
} \hspace{5mm} %
\subfigure[Otherwise]{
\begin{tikzpicture}
\filldraw[color=black!20, fill=black!5, rotate=50] (0.5,0) ellipse (2.1 and 1.1);
\filldraw[black] (1.1, 1.8) circle (0pt) node{$A$};
\filldraw[color=red!20, fill=red!5, thick](1.4, -0.3) circle (0.7);
\filldraw[black] (1.4, -0.3) circle (0pt) node{$\mathcal{T}$};
\end{tikzpicture}
}
\end{center}
\caption{Diagrams representing the three possible states of a realization of a random set $\mathcal{T}$, supported on subsets of $\Theta$, relative to a fixed subset $A$ of $\Theta$.}
\label{fig:rand.set}
\end{figure}
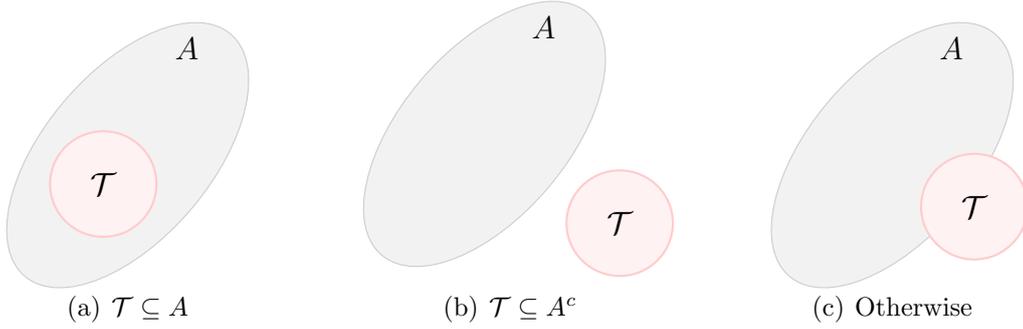

\subsection{Validity property}
\label{SS:valid}

The definition of an IM is too flexible to be useful.  But the degrees of freedom will be reduced considerably once the objective is stated formally.  The IM output will be used to assess the plausibility of various assertions about $\theta$, so it makes sense to require that inferences drawn based on these assessments would not be systematically misleading.  The following definition states this more precisely.  

\begin{defn}
\label{def:valid}
An IM $(x,\model,\ldots) \mapsto \uPi_x$ is {\em valid} if 
\begin{equation}
\label{eq:valid}
\sup_{\theta \in A} \prob_{X|\theta}\{\uPi_X(A) \leq \alpha \} \leq \alpha \quad \text{for all $\alpha \in [0,1]$ and $A \in \A$}. 
\end{equation}
In other words, if assertion $A$ is true, then the random variable\footnote{I'll assume throughout that $x \mapsto \uPi_x(A)$ is Borel measurable for all $A \in \A$.} $\uPi_X(A)$ is stochastically no smaller than $\unif(0,1)$ as a function of $X \sim \prob_{X|\theta}$.  
\end{defn} 

For some intuition about what validity aims to achieve, consider the following line from \citet{reid.cox.2014}:
\begin{quote}
{\em \ldots it is unacceptable if a procedure\ldots of representing uncertain knowledge would, if used repeatedly, give systematically misleading conclusions}.  
\end{quote}
The event ``$\uPi_X(A) \leq \alpha$,'' where the capacity assigns relatively small plausibility to assertion $A$, is one that would give the investigator reason to doubt the truthfulness of $A$.  If, as in \eqref{eq:valid}, the assertion $A$ is true, then this event could lead to misleading conclusion.  By requiring that the probability of such an event be controllable, in a concrete way, through the choice of threshold $\alpha$, those misleading conclusions would not be systematic.  As I argue in \citet{martin.nonadditive}, the choice of threshold ``$\leq \alpha$'' inside the probability statement in \eqref{eq:valid} is basically without loss of generality.  

The ``for all $A \in \A$'' part of \eqref{eq:valid} is important for at least three reasons.  First, for the statistician who develops methods to be used by practitioners off-the-shelf, the responsibility for controlling systematically misleading conclusions falls on his/her shoulders. Therefore, to mitigate their risk, that statistician must be clear about the collection $\A$ for which \eqref{eq:valid} holds.  Even if the method's limits are clearly stated, it is likely that practitioners will try to push those limits, so it's in the statistician's best interest to ensure that \eqref{eq:valid} is achieved for a class $\A$ sufficiently rich that it contains any assertion practitioners might imagine.  This helps ensure that the methods-developing statistician, which is most of us in academia, have skin in the game.  Second, since the class $\A$ is assumed to be closed under complementation, if \eqref{eq:valid} holds for $A$, then it also holds for $A^c$, and it follows that \eqref{eq:valid} is equivalent to 
\[ \sup_{\theta \not\in A} \prob_{X|\theta}\{ \lPi_X(A) \geq 1-\alpha\} \leq \alpha, \quad \text{for all $\alpha \in [0,1]$ and $A \in \A$}. \]
As before, assigning high degree of belief to an assertion that happens to be false may lead to erroneous inference, and the above condition provides control on the frequency of such errors.  Third, if validity as in Definition~\ref{def:valid} holds, then it is relatively straightforward to construct procedures based on the IM's  output that achieves the desirable frequentist properties as discussed above.  

\begin{thm}
\label{thm:freq}
Let $\uPi_x$ be the output from a valid IM, and take any $\alpha \in [0,1]$.   
\begin{enumerate}
\item[{\em (a)}] Consider a testing problem with $H_0: \theta \in A$, for $A \in \A$.  Then the test $T_\alpha(x) = 1\{\uPi_x(A) \leq \alpha\}$ is a size-$\alpha$ test in the sense of \eqref{eq:size}. 
\vspace{-2mm}
\item[{\em (b)}] If $\A$ contains all the singletons, then $C_\alpha(x)$, which has two equivalent expressions
\[ C_\alpha(x) = \begin{cases}
\{ \vartheta \in \Theta: \uPi_x(\{\vartheta\}) > \alpha \} \\
\bigcap \{A \in \A: \lPi_x(A) \geq 1-\alpha\},
\end{cases}
\] 
is a $100(1-\alpha)$\% confidence region for $\theta$ in the sense of \eqref{eq:coverage}. 
\end{enumerate}
\end{thm}

\begin{proof}
Part~(a) follows immediately from Definition~\ref{def:valid}.  For Part~(b), note first that the singleton containment assumption is required in order for $\uPi_x(\{\vartheta\})$ to be well-defined.  Next, it is easy to see, from \eqref{eq:valid} applied to $A = \{\vartheta\}$, that the coverage condition is satisfied.  All that remains to be shown is that the two expressions for $C_\alpha(x)$ are the same.  To see this, note that 
\[ \uPi_x(\{\vartheta\}) > \alpha \iff \lPi_x(\{\vartheta\}^c) < 1-\alpha. \]
By monotonicity of the capacity, it follows that $\lPi_x(A) < 1-\alpha$ for any $A \in \A$ such that $A \subseteq \{\vartheta\}^c$.  Therefore, $\vartheta$ is contained in every $A$ in the intersection defining $C_\alpha(x)$, so it's also in the intersection.   
\end{proof}

To the primary question I intend to answer here: what additional structure must be imposed on the data-dependent capacity  $\uPi_x$ such that the validity condition \eqref{eq:valid} holds?  As a first step, notice that it can be made to hold trivially, by taking $\uPi_x(A) = 1$ for all $x$ and all $A \in \A$; this corresponds to a so-called {\em vacuous belief} model.  Of course, such a choice is not practically useful, and it imposes the maximal structure on the capacity so, while it does technically answer the question, this is not the particular answer I'm looking for.  To investigate this further, I'll start with a frequentist procedure, namely, confidence regions, that already satisfy their own kind of validity property, and see what happens concerning \eqref{eq:valid} when these are used to construct a capacity on $\Theta$.

\section{Confidence and validity}
\label{S:conf.valid}

\subsection{Confidence as probability isn't (really) valid}
\label{SS:isnt.valid}

Confidence is a fundamental concept in statistics, ``arguably the most substantive ingredient in modern model-based theory'' \citep{fraser2011.rejoinder}, dating back to \citet{neyman1941} and also, indirectly, to Fisher through its close ties to fiducial inference \citep{zabell1992, efron1998}.  Despite the uncanny resemblance between confidence and (fiducial) probability, their difference was clear to \citet[][p.~74]{fisher1973}:
\begin{quote}
{\em [Confidence regions] were I think developed and advocated under the impression that in a wider class of cases they could provide information similar to that of the probability statements derived by the fiducial argument.  It is clear, however, that no exact probability statements can be based on them.}
\end{quote}
But since most statisticians can't draw the line between confidence and probability so clearly as Fisher, it's worth asking if the distinction is even necessary: can we ignore the difference and treat confidence as probability?  Below I describe this conversion from confidence to probability and its implications in terms of the property \eqref{eq:valid}.  


Fix data $X=x$ and a family $\{C_\alpha: \alpha \in [0,1]\}$ of confidence regions for $\theta$.  A so-called {\em confidence distribution} \citep[e.g.,][]{schweder.hjort.2002, schweder.hjort.book, xie.singh.2012, nadarajah.etal.2015} corresponds to a countably additive capacity $\Pi_x$ such that 
\[ \Pi_x\{C_\alpha(x)\} = 1-\alpha, \quad \text{for all $\alpha \in [0,1]$ and all $x$}. \]
I've removed the upper-line on the capacity notation to emphasize that this is a probability measure.  The simplest case is where the family $\{C_\alpha: \alpha \in [0,1]\}$ consists of confidence upper limits, so that the parametric curve $\{(C_\alpha(x), 1-\alpha): \alpha \in [0,1]\}$ defines a distribution function that determines $\Pi_x$.  More generally, $C_\alpha(x)$ is a level set for the confidence density which, together with the above display, determines the entire distribution $\Pi_x$.  In any case, this confidence distribution approach boils down to interpreting confidence regions in the way that students in introductory statistics courses are told not to, i.e., ``$\theta$ is in $C_\alpha(x)$ with probability $1-\alpha$.''  This is a subjective probability that the data analyst is assigning, so it can't be ``wrong'' in principle.  The question, however, is if representing confidence as a probability distribution achieves the goals of probabilistic inference.  

A major selling point for those frameworks that base inference on a distribution is that one can apply the probability calculus to derive an answer to any relevant question.  For example, \citet[][p.~5]{xie.singh.2012} write: ``[a confidence distribution] contains a wealth of information for inference; much more than a point estimator or a confidence interval.'' Presumably, part of this extra information is the ability to assign probabilities to all sorts of assertions, but the value in this lies in whether it can be used reliably across applications.  In particular, is the confidence distribution valid in the sense of Definition~\ref{def:valid} for a sufficiently rich collection $\A$ of assertions?  

Consider arguably the simplest scalar parameter example, namely, $X \sim \nm(\theta,1)$, and take $C_\alpha(x) = x + z_{1-\alpha}$, where $z_{1-\alpha}$ is the $(1-\alpha)$-quantile of $\nm(0,1)$.  Then it's easy to check that $\Pi_x = \nm(x,1)$ is a confidence distribution in the sense of Definition~1 in \citet{xie.singh.2012}, that is, 
\[ \Pi_X\{ (-\infty, \theta] \} \sim \unif(0,1), \quad \text{as a function of $X \sim \nm(\theta,1)$, for all fixed $\theta$}. \]
This implies validity in the sense of Definition~\ref{def:valid} above, for a collection $\A$ that contains the half-lines $\{(-\infty, t]: t \in \RR\}$.  But what about other kinds of assertions besides half-lines?  After all, the confidence distribution was defined via a collection of confidence upper limits, or half-lines, so validity only for half-lines arguably hasn't added anything beyond the confidence limits themselves.  An advantage of working with probability is that there is a familiar procedure to go from the distribution of $\theta$ to that of $\phi=\phi(\theta)$.  As an example, consider $\phi = |\theta|$.  It's straightforward to derive the corresponding confidence distribution for $\phi$ via the textbook probability calculus.  However, {\em that derived distribution for $\phi$ does not meet the conditions of a confidence distribution}. Indeed, Figure~\ref{fig:abstheta} plots the distribution function 
\begin{equation}
\label{eq:cd.cdf}
G_\theta(\alpha) = \prob_{X|\theta}\bigl[ \Pi_X(\{\vartheta: |\vartheta| \leq |\theta|\}) \leq \alpha \bigr], \quad \alpha \in [0,1], 
\end{equation}
for the case when $\theta=0.5$.  According to Definition~1 in \citet{xie.singh.2012}, in order for this marginal ``confidence distribution'' of $\phi$ to be a genuine  confidence distribution, $\alpha \mapsto G_\theta(\alpha)$ must be a $\unif(0,1)$ distribution function.  Clearly it's not.  Therefore, if one starts with a genuine confidence distribution and applies the ordinary probability calculus, then the result need not be a confidence distribution.  

\begin{figure}[t]
\begin{center}
\scalebox{0.70}{\includegraphics{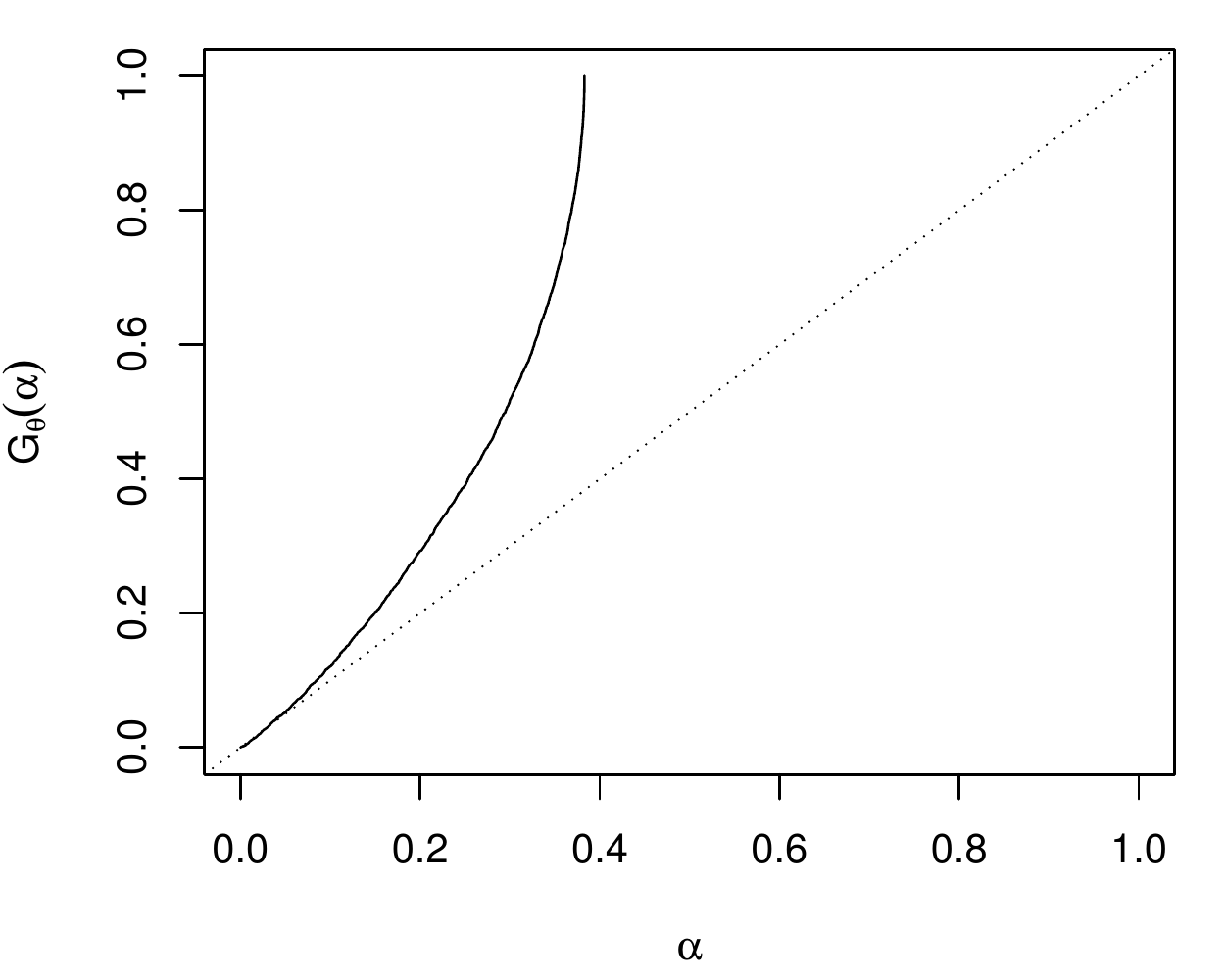}}
\end{center}
\caption{Plot of the distribution function (solid) $\alpha \mapsto G_\theta(\alpha)$ as defined in \eqref{eq:cd.cdf}, with $\theta=0.5$.  According to Definition~1 in \citet{xie.singh.2012}, this solid line needs to be on the dotted diagonal line in order to be a confidence distribution.}
\label{fig:abstheta}
\end{figure} 

In terms of validity from Section~\ref{SS:valid}, the fact that the distribution function for $\phi=|\theta|$ in Figure~\ref{fig:abstheta} is above the diagonal line means that the confidence distribution fails to be valid in the sense of Definition~\ref{def:valid} if $\A$ contains certain bounded intervals.  If $\A$ can't contain bounded intervals, there there is good reason to believe that the biggest collection $\A$ for which the confidence distribution is valid is the set of half-lines, which is basically just the collection of confidence upper limits I started with.  So, if I insist that the features derived from $\Pi_x$ be reliable in the sense of \eqref{eq:valid}, then apparently the confidence distribution contains no additional information beyond what was already contained in the original collection of confidence limits.  

The situation is even more problematic when the parameter is multi-dimensional, so much so that it's even been recommended not to attempt constructing them:
\begin{quote}
{\em joint [confidence distributions] should not be sought, we think, since they might easily lead the statistician astray} \citep[][p.~59]{schweder.hjort.cd.discuss}.
\end{quote}
While confidence distributions are not recommended when $\theta$ is multi-dimensional, an illustration is needed both to clarify what Schweder and Hjort mean by ``lead the statistician astray'' and for comparison with what's coming in Section~\ref{SS:is.valid}.  Consider the so-called Fieller--Creasy problem \citep{fieller1954, creasy1954} which, in its simplest form, concerns inference about the ratio $\phi = \theta_1/\theta_2$ based on two independent normal observations, $X_1 \sim \nm(\theta_1,1)$ and $X_2 \sim \nm(\theta_2, 1)$.  The only option for a joint confidence distribution for $\theta=(\theta_1,\theta_2)$, based on $x=(x_1,x_2)$, is $\Pi_x=\nm_2(x, I_2)$, an independent bivariate normal with mean $x$ and unit variances; this is also Fisher's fiducial distribution and Jeffreys's ``default-prior'' Bayesian posterior.  The marginal distribution $\Pi_x^\phi$ for $\phi = \theta_1/\theta_2$ can be derived as in \citet{hinkley1969}, but it's easier to approximate the probabilities using Monte Carlo. 
In particular, for an assertion $B = (-\infty,9]$ concerning $\phi$, I can simulate pairs $(\vartheta_1,\vartheta_2)$ from $\Pi_x$ and approximate $\Pi_x^\phi(B)$ by the proportion of those simulated pairs satisfy $\vartheta_1/\vartheta_2 \leq 9$.  As before, I'm interested in the distribution function 
\begin{equation}
\label{eq:cd.cdf.2}
G_\theta(\alpha) = \prob_{X|\theta}\bigl[ \Pi_X^\phi(B) \leq \alpha \bigr], \quad \alpha \in [0,1]. 
\end{equation}
Figure~\ref{fig:fc} shows a plot of the distribution function in \eqref{eq:cd.cdf.2} when $\theta=(10,1)$.  In this case, $\phi=10$ and the assertion $B$ about $\phi$ is false.  One would, therefore, expect that the probability assigned to $B$ would tend to be small.  However, as is clear from the plot, the $\Pi_X(B)$ values are {\em almost always close to 1}, hence false confidence.  
One possible explanation for the false confidence phenomenon in this case is that the Fieller--Creasy problem is one member of the surprisingly challenging class of examples described by \citet{gleser.hwang.1987}; see, also, \citet{dufour1997}.  What makes this class of examples special is that any confidence interval having almost surely finite length necessarily has coverage probability equal to 0.  Note that intervals derived from quantiles of a {\em distribution} can't be unbounded.  Therefore, the confidence ``distribution'' derived in, e.g., \citet[][p.~61]{schweder.hjort.cd.discuss}, for the Fieller--Creasy problem, which yields intervals for $\phi$ with exact coverage probability, can't be a genuine distribution; see Section~\ref{SS:is.valid}.  

\begin{figure}[t]
\begin{center}
\scalebox{0.70}{\includegraphics{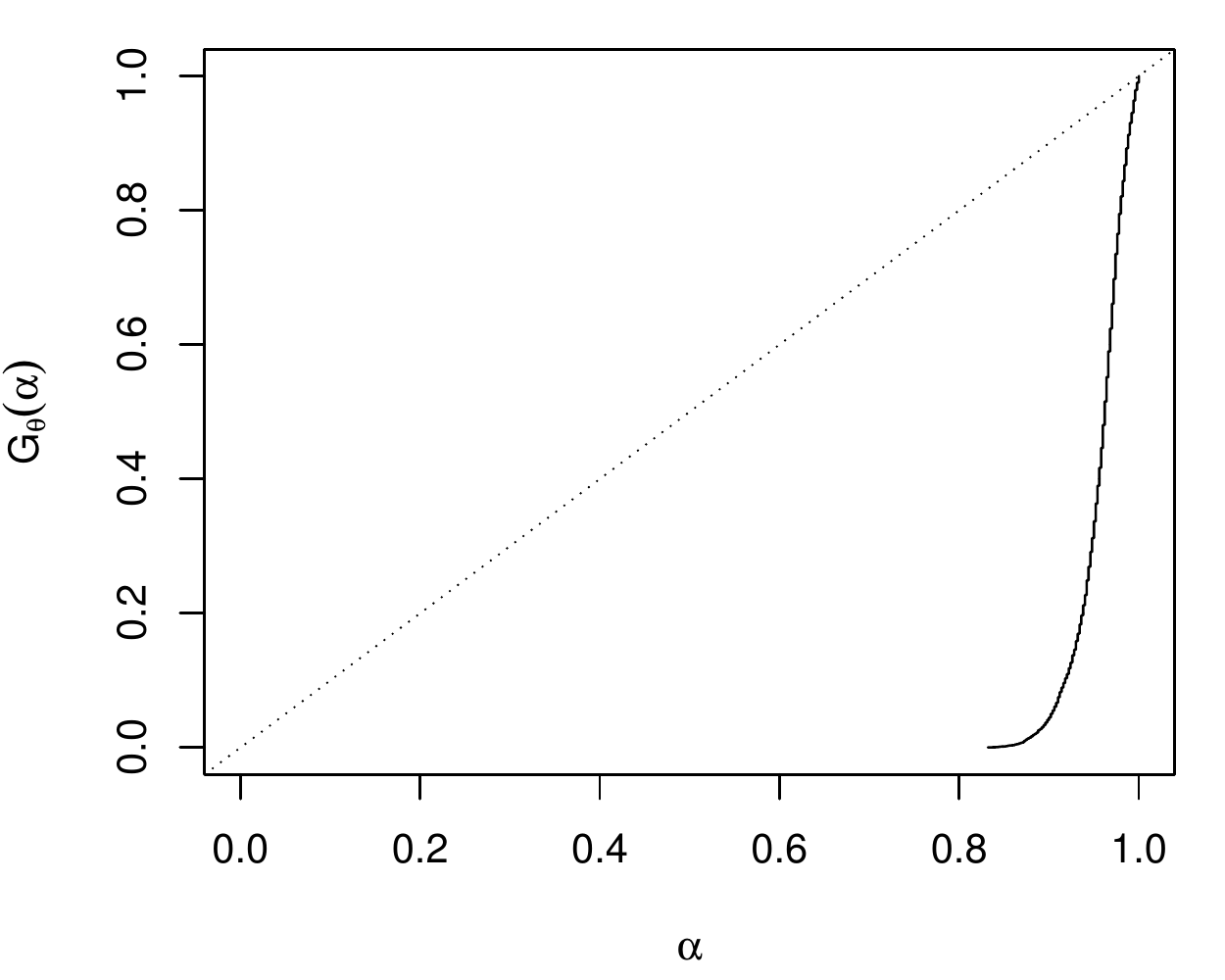}}
\end{center}
\caption{Plot of the distribution function (solid) $\alpha \mapsto G_\theta(\alpha)$ as defined in \eqref{eq:cd.cdf.2}, with $\theta=(1,0.1)$.  The distribution function pushed to the right implies $\Pi_X^\phi(B)$ tends to be large, even though $B$ is a false assertion.}
\label{fig:fc}
\end{figure} 

Much of what I said above is known.  For years, Fraser has been warning about the risk associated with using probability in place of confidence; some recent papers are \citet{fraser2011, fraser2011.rejoinder, fraser.cd.discuss, fraser.copss, fraser2018} and \citet{fraser.etal.2016}.  The notion of validity, which plays a significant role in what follows, provides a new angle from which these problems can be viewed and, hopefully, resolved.  Indeed, the following {\em false confidence theorem} \citep{balch.martin.ferson.2017} highlights the disconnect between the strong frequentist notion of validity and the additivity property of probability distributions. 

\begin{thm}[False Confidence]
\label{thm:fct}
Let $\Pi_X$ be a continuous data-dependent probability distribution, e.g., Bayesian, fiducial, confidence distribution, etc.  Then for any $\theta \in \Theta$, any $\alpha \in [0,1]$, and any increasing function $g: [0,1] \to [0,1]$, there exists $A \subset \Theta$ such that 
\[ A \ni \theta \quad \text{and} \quad \prob_{X|\theta}\{\Pi_X(A) \leq g(\alpha) \} > \alpha. \]
\end{thm}

In words, the false confidence theorem says that there exists true assertions that tend to be assigned relatively small $\Pi_X$-probability; of course, this can be equivalently stated in terms of assigning relatively large probabilities to false assertions, which is the version presented in \citet{balch.martin.ferson.2017} and explains the name ``false confidence.''  Since assigning relatively small (resp.~large) probability to true (resp.~false) assertions creates opportunities for misleading inferences, this should be avoided.  The theorem states that {\em any data-dependent probability distribution} suffer from this affliction, so it's not an issue of the Bayesian's prior, how the fiducial or confidence distribution is constructed, etc., it's a fundamental shortcoming of probability as a quantification of uncertainty in the context of statistical inference.  Short of cutting down the collection $\A$ of candidate assertions in a non-trivial way (see Section~\ref{SS:rich}), the only way to completely avoid false confidence is to consider imprecise probabilities, as I discuss next.

\subsection{Confidence as plausibility is valid}
\label{SS:is.valid}

Recall the standard connection between confidence regions and significance tests.  In particular, for a test that rejects the null hypothesis $H_0: \theta=\vartheta$ in favor of the alternative $H_1: \theta \neq \vartheta$, at level $\alpha \in [0,1]$, if $C_\alpha(x) \not\ni \vartheta$, the p-value is given by 
\begin{equation}
\label{eq:cr.contour}
\pi_x(\vartheta) = \sup\{\alpha \in [0,1]: C_\alpha(x) \ni \vartheta\}, \quad \vartheta \in \Theta. 
\end{equation}
Allowing $\vartheta$ to vary determines a ``p-value function'' \citep{pvalue.course} that has many other names, including {\em confidence curve} \citep{birnbaum1961, schweder.hjort.book}, {\em possibility function} \citep{zadeh1978, dubois2006}, {\em preference function} \citep{spjotvoll1983, blaker.spjotvoll.2000}, and {\em significance function} \citep{fraser1990, fraser1991}, among others.  Here, I adopt the terminology {\em plausibility contour} for two reasons: first, it aligns with the natural interpretation of the output derived from frequentist procedures and, second, there is a formal mathematical theory available for describing these objects.  

Assume that the collection $\{C_\alpha: \alpha \in [0,1]\}$ is {\em nested} in the sense that
\begin{equation}
\label{eq:nested.C}
\text{for all $x \in \XX$} \quad \begin{cases} 
\; \alpha \geq \alpha' \implies C_\alpha(x) \subseteq C_{\alpha'}(x) \\
\; \textstyle\bigcup_\alpha C_\alpha(x) = \Theta \quad \text{and} \quad \textstyle \bigcap_\alpha C_\alpha(x) \neq \varnothing. 
\end{cases}
\end{equation}
The latter non-emptiness condition amounts to there existing a point, say, $\hat\theta(x)$ that belongs to every confidence region $C_\alpha(x)$, which is a standard relationship between point estimators and confidence regions.  For example, the Bayesian maximum {\em a posteriori} or MAP estimator belongs to all of the highest posterior density credible regions.  Note that it is not necessary for the intersection in \eqref{eq:nested.C} to be a single point; see Section~\ref{SS:binomial}.  Also, in cases where $C_\alpha(x)$ is ``one-sided,'' like a confidence lower/upper bound, then $\hat\theta(x)$ can be infinite. In general, nested $C_\alpha$ implies that $\pi_x$ in \eqref{eq:cr.contour} satisfies 
\begin{equation}
\label{eq:supremum}
\sup_{\vartheta \in \Theta} \pi_x(\vartheta) = 1. 
\end{equation}
From here, one can construct a capacity on $(\Theta,\A)$, with $\A = 2^\Theta$, via the rule
\begin{equation}
\label{eq:plaus.consonance}
\uPi_x(A) = \sup_{\vartheta \in A} \pi_x(\vartheta), \quad A \subseteq \Theta. 
\end{equation}
This capacity satisfies the properties of a {\em consonant plausibility function} \citep[e.g.,][]{shafer1976, shafer1987, balch2012} or, equivalently, a {\em possibility measure} \citep[e.g.,][]{zadeh1978, dubois.prade.1987, dubois.prade.book}.  Here I'll refer to the quantity in \eqref{eq:plaus.consonance} as a {\em plausibility measure}.  As \citet[][p.~226]{shafer1976} explains, consonance isn't an appropriate model for all kinds of evidence, but it's ``well adapted'' for statistical inference problems, where ``evidence for a cause is provided by an effect.'' 

There is a rich mathematical theory associated with plausibility/possibility measures, most of which will not be needed here.  I refer the interested reader to, e.g., the series of papers \citet{cooman.poss1, cooman.poss2, cooman.poss3}.  The final technical comment I'd like to make about plausibility/possibility measures in general is that these are among the most simple imprecise probability models, that is, they have a lot of structure.  Specifically, while the plausibility measure is a genuine set function, it's completely determined by its values on singleton sets or, in other words, it's completely determined by its contour function \eqref{eq:plaus.consonance}.  Compare this to traditional theory where a probability measure, a complex set function, is expressed as the integral of a density, a point function; the difference here is that integration is replaced by optimization. 

Two immediate insights about the fixed-$x$ interpretation of the confidence region $C_\alpha(x)$ emerge based on the derived belief and plausibility measures.  First, note that, by definition of $\pi_x$ in \eqref{eq:cr.contour}, 
\[ C_\alpha(x) = \{\vartheta \in \Theta: \pi_x(\vartheta) > \alpha\}. \]
Therefore, there is a mathematical explanation for the interpretation I gave in Section~\ref{S:intro} of a confidence region as a set of parameter values that are individually sufficiently plausible.  Alternatively, and again from the above definitions, 
\[ \lPi_x\{C_\alpha(x)\} = 1-\alpha, \]
which implies that an equivalent interpretation of $C_\alpha(x)$, for given $x$, is as a set of parameter values which are, collectively, sufficiently believable.  The former interpretation is more in line with a hypothesis testing mode of thinking, i.e., any parameter value that's inside $C_\alpha(x)$ is plausible.  The latter is perhaps more natural, and consistent with what confidence distributions aimed to achieve, but different from the familiar textbook way of thinking which is based on hypothesis testing. 

The belief and plausibility measures described above are well-defined and have their own mathematical properties, but the words ``belief'' and ``plausibility'' also have meaning in everyday language.  So even though $C_\alpha(x)$ is, mathematically, a set of parameter values whose point plausibility exceeds $\alpha$, what is my justification for concluding that the points $C_\alpha(x)$ contains are plausible in the everyday sense?  The justification comes from the statistical properties that the confidence region satisfies.  That is, by definition, 
\begin{equation}
\label{eq:pval.unif}
\prob_{X|\theta}\{C_\alpha(X) \not\ni \theta\} = \prob_{X|\theta}\{ \pi_X(\theta) \leq \alpha \} \leq \alpha. 
\end{equation}
Since the event ``$\pi_X(\theta) \leq \alpha$'' is rare (in a precise sense), I have reason to treat any value $\vartheta$ that satisfies $\pi_x(\vartheta) > \alpha$ as sufficiently plausible.  Similarly, since the event ``$C_X(\alpha) \not\ni \theta$'' is rare, I have reason to treat the collection of $\vartheta$ that satisfy $\vartheta \in C_\alpha(x)$ as sufficiently believable.  This line of reasoning is consistent with the {\em reliabilist} perspective on justified beliefs in epistemology \citep[e.g.,][]{goldman1979}.  Note the difference between my approach and that of introductory textbooks: the latter attempt to interpret confidence regions through its properties, whereas the former gives confidence regions a natural, everyday language interpretation which is justified by their statistical properties.

The consonance property created by the definition \eqref{eq:plaus.consonance} turns out to be critical for validity.  Indeed, consonance immediately implies 
\[ \{\text{$\theta \in A$ and $\uPi_X(A) \leq \alpha$}\} \implies \pi_X(\theta) \leq \alpha. \]
From the assumed coverage probability property \eqref{eq:coverage} of the given family of confidence regions, and its reinterpretation through the equality in \eqref{eq:pval.unif}, it is easy to see that 
\begin{equation}
\label{eq:plaus.valid}
\sup_{\theta \in A} \prob_{X|\theta} \{\uPi_X(A) \leq \alpha\} \leq \alpha, \quad \text{for all $A \subseteq \Theta$ and all $\alpha \in [0,1]$}. 
\end{equation}
This proves the following elementary yet fundamental result, a bit stronger than what's referred to in \citet{walley2002} as the {\em fundamental frequentist principle}. 

\begin{thm}
\label{thm:consonance.valid}
Let $\{C_\alpha: \alpha \in [0,1]\}$ be a family of confidence regions satisfying \eqref{eq:coverage} whose corresponding contour function in \eqref{eq:cr.contour} satisfies \eqref{eq:supremum}.  Then the IM with capacity $\uPi_x$ being the plausibility measure in \eqref{eq:plaus.consonance} is valid in the sense of Definition~\ref{def:valid} for $\A=2^\Theta$. 
\end{thm}

This result confirms the claim in this subsection's heading that confidence, as interpreted as a plausibility measure satisfying the consonance property in \eqref{eq:plaus.consonance} is valid.  Being valid on its own may not be so meaningful, but being valid with {\em no restrictions on $\A$} is important.  To be clear, the reason I emphasize ``no restrictions on $\A$'' is not because I care about weird/non-measurable assertions.  Rather, in order to not ``lead the statistician astray,'' it's practically important to be able to reliably marginalize to any feature $\phi=\phi(\theta)$ of interest.  Converting confidence into a probability distribution is unable to achieve this, but  Theorem~\ref{thm:consonance.valid} says that it can be achieved by converting confidence to a (consonant) plausibility measure.  To see this in action, let $\phi$ be an interest parameter, taking values in $\Phi$, and $B \in 2^\Phi$ a relevant assertion about $\phi$.  Define a marginal IM for $\phi$ with plausibility measure 
\[ \uPi_x^\phi(B) = \uPi_x(\{\vartheta \in \Theta: \phi(\vartheta) \in B\}) = \sup_{\vartheta: \phi(\vartheta) \in B} \pi_x(\vartheta). \]
If I define $A=\{\vartheta: \phi(\vartheta) \in B\}$, then clearly $A \in \A$ and validity of the marginal IM for $\phi$ follows from Theorem~\ref{thm:consonance.valid}.  This proves the following

\begin{cor0}
The marginal IM for $\phi=\phi(\theta)$ derived from the valid IM for $\theta$ defined by \eqref{eq:plaus.consonance} is also valid for all marginal assertions $B$ about $\phi$.  
\end{cor0}

The style of marginalization described above is more familiar than it might seem.  Given a confidence region $C_\alpha(x)$ for $\theta$, the natural way to construct a corresponding confidence region for $\phi=\phi(\theta)$ is to just look at the image of $C_\alpha(x)$ under the mapping $\phi$.  That is, the corresponding confidence region for $\phi$ is 
\[ C_\alpha^\phi(x) = \{\varphi: \phi(\vartheta) = \varphi \text{ for some } \vartheta \in C_\alpha(x)\}. \]
It's not difficult to see that, if the marginal plausibility contour is 
\[ \pi_x^\phi(\varphi) = \sup_{\vartheta: \phi(\vartheta) = \varphi} \pi_x(\vartheta), \quad \varphi \in \Phi, \]
then the collection of all sufficiently plausible values of $\phi$ satisfies 
\[ \{\varphi \in \Phi: \pi_x^\phi(\varphi) > \alpha\} = C_\alpha^\phi(x). \]
Therefore, marginalization via the formal and mathematically rigorous consonant plausibility calculus agrees with the intuitive marginalization strategy one adopts when manipulating confidence regions---and it preserves validity of the original IM for $\theta$.  

For a quick illustration, reconsider the independent normal means example as in Section~\ref{SS:isnt.valid} above, with $X_1 \sim \nm(\theta_1, 1)$ and $X_2 \sim \nm(\theta_2, 1)$, leading up to the Fieller--Creasy problem.  The joint confidence regions for $\theta=(\theta_1,\theta_2)$ are disks centered at $x=(x_1,x_2)$, and the underlying normality implies that the corresponding plausibility contour for $\theta$ is 
\[ \pi_x(\vartheta) = 1 - G_2(\|\vartheta-x\|_2^2), \quad \vartheta = (\vartheta_1, \vartheta_2) \in \RR^2, \]
where $G_d$ is the $\chisq(d)$ distribution function with $d$ degrees of freedom.  To check that this is indeed the correct plausibility contour, note that $\{\vartheta: \pi_x(\vartheta) > \alpha\}$ is the usual confidence disk in $\RR^2$.  First, suppose that $\phi=\theta_1$ is a quantity of interest, then the marginal plausibility contour is
\[ \pi_x^\phi(\varphi) = \sup_{\vartheta_2 \in \RR} \pi_x\bigl( (\varphi, \vartheta_2) \bigr) = \pi_x\bigl( (\varphi, x_2) \bigr), \quad \varphi \in \RR. \]
A plot of this marginal plausibility contour is shown in Figure~\ref{fig:fc2}.  If it were known in advance that $\theta_2$ were totally irrelevant and only assertions about $\theta_1$ would be considered, then there is a more direct and efficient marginalization strategy available.  That is, if I just ignored $X_2$ and $\theta_2$ entirely, and just constructed a (marginal) plausibility contour for $\phi=\theta_1$ based on $X_1$ alone.  In that case, 
\[ \pi_x^\phi(\varphi) = 1 - | 2 F_\varphi(x_1) - 1 |, \quad \varphi \in \RR, \]
where $F_\varphi$ is the $\nm(\varphi, 1)$ distribution function.  This is also plotted in Figure~\ref{fig:fc2}.  Both lead to valid inference on $\phi=\theta_1$, but the latter having tighter plausibility contours compared to the former is a reflection of the efficiency gains that are possible due to direct marginalization.  This is not an instance of a free lunch---the gain of efficiency comes at the cost of not being able to make inference on quantities related to $\theta_2$.  

\begin{figure}[t]
\begin{center}
\scalebox{0.70}{\includegraphics{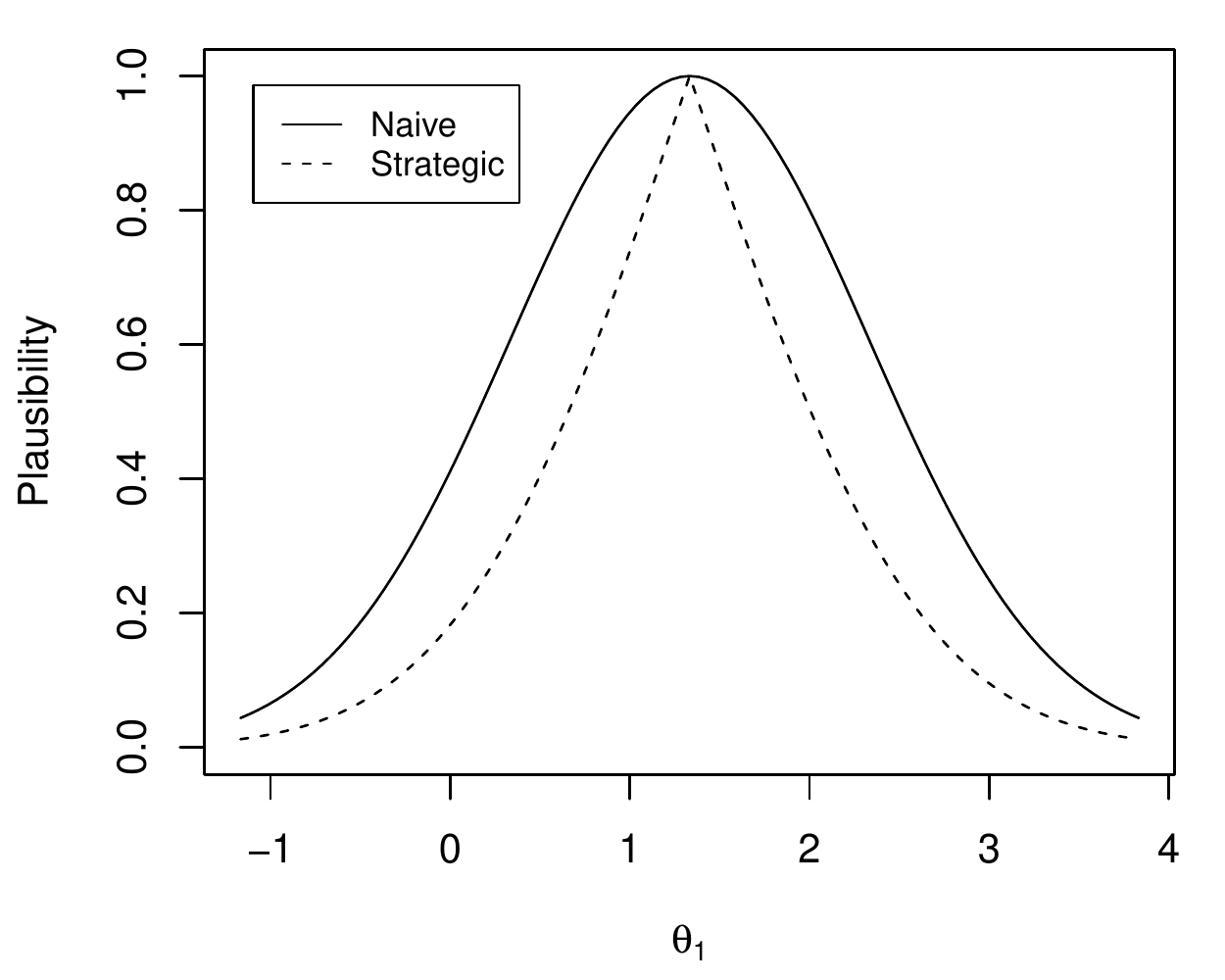}}
\end{center}
\caption{Plots of two marginal plausibility contours for $\phi=\theta_1$ in the normal means model with data $X=(1.333, 0.333)$. The ``naive'' approach (solid) is based on marginalizing the joint plausibility contour using consonance, while the ``strategic'' approach (dashed) is that based on ignoring $\theta_2$ from the start.}
\label{fig:fc2}
\end{figure} 

Second, for the Fieller--Creasy problem, the interest parameter is  $\phi=\theta_1/\theta_2$, and the corresponding marginal plausibility contour is 
\[ \pi_x^\phi(\varphi) = \sup_{z \in \RR} \pi_x\bigl( (\varphi z, z) \bigr), \]
and it's not too difficult to show that $\pi_x^\phi(\varphi) = \pi_x(\vartheta^{x,\varphi})$, where 
\[ \vartheta^{x,\varphi} = (1 + \varphi^2)^{-1} \bigl(\varphi^2 x_1 + \varphi x_2 \, , \, \varphi x_1 + x_2 \bigr)^\top. \]
The first question is about validity, and here I reproduce the brief simulation summarized in Figure~\ref{fig:fc}.  That is, I evaluate the lower probability $\uPi_x^\phi(B)$ assigned to the assertion $B=(-\infty,9]$ using the consonant plausibility calculus.  Then I evaluate the distribution function of $\lPi_X^\phi(B)$, as a function of $X \sim \prob_{X|\theta}$, in the case where $\theta=(1,0.1)$, which makes the assertion false.  As expected from the validity result, this distribution function is above the diagonal line in Figure~\ref{fig:fc3}(a), whereas the analogous distribution function in \eqref{eq:cd.cdf.2}, for the additive $\Pi_X^\phi$, is well below the diagonal line.  The second question concerns efficiency, and a plot of the marginal plausibility contour for $\phi$ is shown in Figure~\ref{fig:fc3}(b).  Alternatively, it is possible to derive a confidence region for $\phi$ by directly marginalizing first before constructing a confidence distribution.  The plausibility contour based on the marginal confidence distribution derived in \citet{schweder.hjort.cd.discuss} is given by 
\[ \pi_x^\phi(\varphi) = 1 - G_1\bigl\{ (1+\varphi^2)^{-1}(x_1 - \varphi x_2)^2 \bigr\}, 
\quad \varphi \in \RR. \]
The same marginal plausibility contour was derived in \citet[][Sec.~4.2]{immarg} using a different argument.  A plot of this contour is also displayed in Figure~\ref{fig:fc3}(b) and, since I used the same data as in \citet{schweder.hjort.cd.discuss}, in particular, $X=(1.333, 0.333)$, my plot is the same as in their Figure~1.\footnote{Schweder and Hjort (and others) often plot a ``confidence curve'' which is 1 minus my plausibility contour.  So their plot is actually a reflection of mine about the horizontal line at 0.5.} Here, again, note the gain in efficiency between doing the naive marginalization of the joint plausibility contour compared to marginalizing more strategically to $\phi$ from the start.  

\begin{figure}[t]
\begin{center}
\subfigure[Distribution functions of $\Pi_X^\phi(B)$ and $\lPi_X^\phi(B)$]{\scalebox{0.60}{\includegraphics{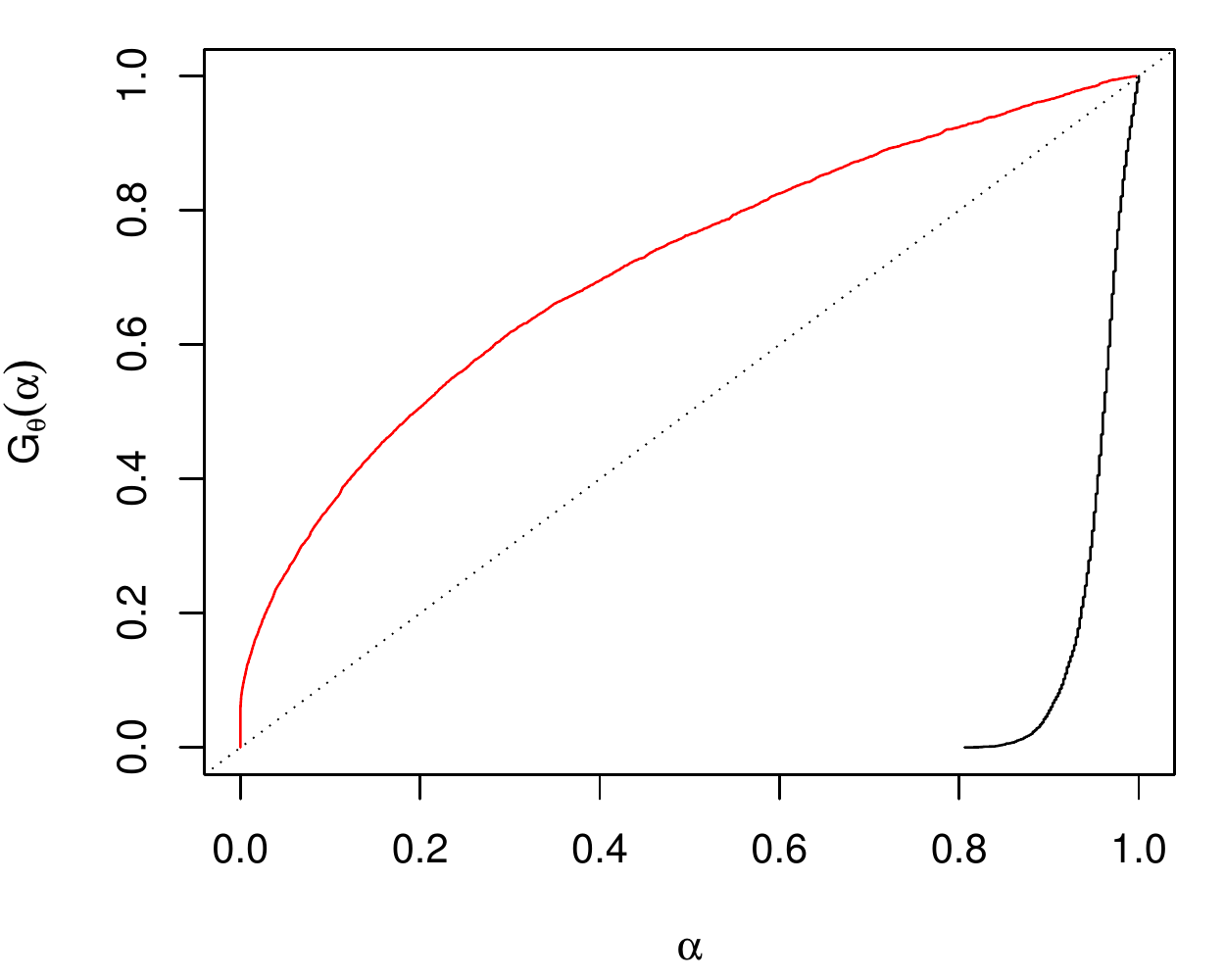}}}
\subfigure[Marginal plausibility contours]{\scalebox{0.60}{\includegraphics{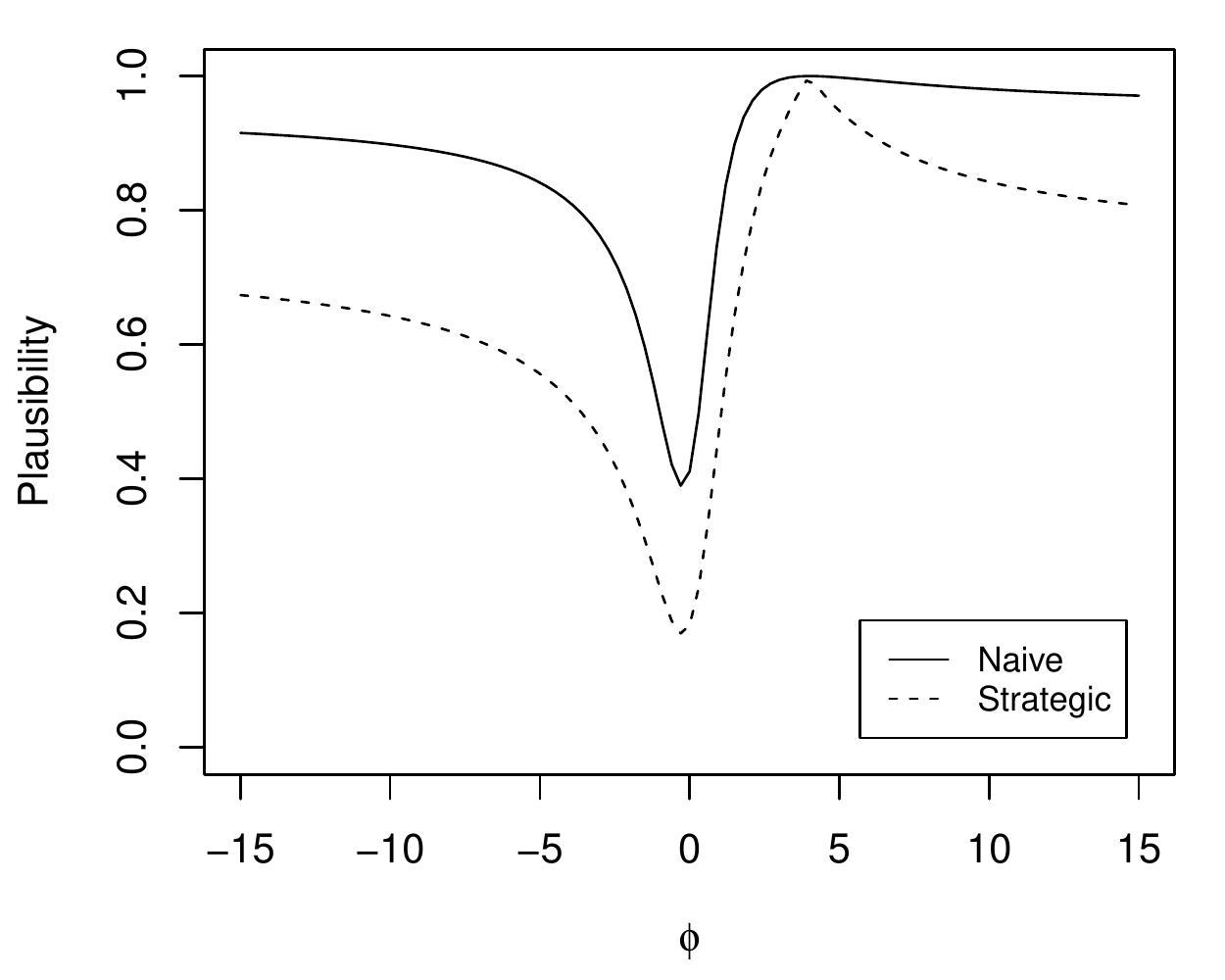}}}
\end{center}
\caption{Panel~(a): Distribution functions of $\Pi_X^\phi(B)$ (black) and of $\lPi_X^\phi(B)$ (red) compared to $\unif(0,1)$.  Panel~(b): Plots of two marginal plausibility contours for $\phi=\theta_1/\theta_2$ in the normal means model with data $X=(1.333, 0.333)$.  The ``naive'' approach (solid) is based on marginalizing the joint plausibility contour using consonance, while the ``strategic'' approach (dashed) is based on direct marginalization to $\phi$ from the start.}
\label{fig:fc3}
\end{figure} 

To summarize, while an interpretation of confidence in terms of probability is tempting, this generally doesn't work.  That is, in order for a confidence distribution---or any other data-dependent probability distribution for that matter---to be valid, the class of assertions needs to be severely restricted, perhaps to the point that the distribution is no more informative than the family of confidence regions one started with.  However, it's straightforward and in many ways intuitive to, instead, interpret confidence in terms of imprecise probability, namely, as a consonant plausibility.  This has (at least) the following three practical benefits:
\begin{itemize}
\item from the point of view of plausibility theory, confidence regions (and p-values) are easy to interpret, and that interpretation naturally aligns with how practitioners already interpret them; 
\vspace{-2mm}
\item the coverage probability property \eqref{eq:coverage} together with consonant plausibility construction easily and directly leads to validity with no restrictions on the class $\A$ of assertions; and 
\vspace{-2mm}
\item the mathematically rigorous plausibility theory comes with its own calculus for manipulating confidence-as-plausibility, e.g., for marginalization, and following those formal rules preserves the validity property in the previous bullet.  
\end{itemize} 
Finally, it's important to emphasize that the imprecise probability perspective presented here is not simply an embellishment on top of confidence regions or p-values.  It's possible to construct valid IMs directly, without already having a confidence region in hand.  I'll discuss this construction next in Section~\ref{S:source}.  This machinery for constructing valid IMs will lead to an even stronger characterization of frequentest procedures in terms of imprecise probabilities in Section~\ref{S:main}, with important practical implications.  

\section{Direct construction of valid IMs}
\label{S:source}

\subsection{Using one random set}
\label{SS:one.S}

Towards a stronger characterization of frequentist procedures in terms of valid IMs, here I consider the question of how to directly construct a valid IM having the plausibility measure structure developed in the previous section, where by ``directly'' I mean without starting from a given frequentist procedure.  To my knowledge, the only direct construction is that in \citet{imbasics, imbook}, described below.  The chief novelty of this approach is the user-specified random set designed to predict or guess the unobserved value of an auxiliary variable.  Here is Martin and Liu's three-step construction. 

\begin{astep}
Associate data $X$ and parameter $\theta$ with an auxiliary variable $U$, taking values in $\UU$, with \emph{known} distribution $\prob_U$.  In particular, let 
\begin{equation}
\label{eq:basic.assoc}
X = a(\theta, U), \quad U \sim \prob_U. 
\end{equation}
This is just a mathematical description of an algorithm for simulating from $\prob_{X|\theta}$, but it's not necessary to assume that data is actually generated according to this process.  The inferential role played by the association is to shift primary focus from $\theta$ to $U$.  To see this, define the sets 
\begin{equation}
\label{eq:basic.focal}
\Theta_x(u) = \{\theta: x = a(\theta,u)\}, \quad u \in \UU, \quad x \in \XX. 
\end{equation} 
Given data $X=x$, if an {\em oracle} tells me a value $u^\star$ that satisfies $x=a(\theta,u^\star)$, then my inference is $\theta \in \Theta_x(u^\star)$, the ``strongest possible'' conclusion.  
\end{astep}

\begin{pstep}
Predict the unobserved value of $U$ in \eqref{eq:basic.assoc} with a (nested) random set $\S$.  This step is motivated by the shift of focus from $\theta$ to $U$, and is the feature that distinguishes the IM framework from Fisher's fiducial and the extensions developed by \citet{dempster1966, dempster1967, dempster1968b, dempster1968a, dempster.copss}.  The distribution $\prob_\S$ of $\S$ is to be chosen by the data analyst, subject to certain conditions; see Theorem~\ref{thm:valid} below. 
\end{pstep}

\begin{cstep}
Combine the association map in \eqref{eq:basic.focal} and the observed data $X=x$ with the random set $\S$ to obtain a new random set
\begin{equation}
\label{eq:post.focal}
\Theta_x(\S) = \bigcup_{u \in \S} \Theta_x(u). 
\end{equation}
Note that $\Theta_x(\S)$ contains the true $\theta$ if (and often only if) $\S$ likewise contains the unobserved value $u^\star$ of $U$.  The IM output is the distribution of $\Theta_x(\S)$, which, if $\Theta_x(\S)$ is non-empty with $\prob_\S$-probability~1 for each $x$, can be summarized by a plausibility measure
\[ \uPi_x(A) = \prob_\S\{\Theta_x(\S) \cap A \neq \varnothing\}, \quad A \subseteq \Theta. \]
\end{cstep}

It turns out that the validity property \eqref{eq:valid} is determined by properties of the random set $\S$.  The following result \citep{imbasics} shows that validity holds for a wide class of random sets $\S$ introduced in the P-step.  Let $(\UU,\U)$ be the measurable space on which $\prob_U$ is defined, and assume that $\U$ contains all closed subsets of $\UU$.  

\begin{thm}
\label{thm:valid}
Suppose that the predictive random set $\S$, supported on $\SS$, with distribution $\prob_\S$, satisfies the following:
\begin{itemize}
\item[\rm P1.] The support $\SS \subset 2^\UU$ contains $\varnothing$ and $\UU$, and: \\
{\rm (a)} is closed, i.e., each $S \in \SS$ is closed and, hence, in $\U$, and \\
{\rm (b)} is nested, i.e., for any $S,S' \in \SS$, either $S \subseteq S'$ or $S' \subseteq S$.
\item[\rm P2.] The distribution $\prob_\S$ satisfies $\prob_\S\{\S \subseteq K\} = \sup_{S \in \SS: S \subseteq K} \prob_U(S)$, for each $K \subseteq \UU$.  
\end{itemize}
In addition, if $\Theta_x(\S)$ is non-empty with $\prob_\S$-probability~1 for all $x$, then the IM is valid in the sense of Definition~\ref{def:valid}.  
\end{thm}

Constructing a random set $\S$ that satisfies P1--P2 is relatively easy; see Corollary~1 in \citet{imbasics}.  Note that the nested support property P1(b) is what makes the plausibility measure consonant, which was fundamental to the discussion in Section~\ref{SS:is.valid}.  The non-emptiness condition, i.e., $\Theta_x(\S) \neq \varnothing$ with $\prob_\S$-probability~1, is easy to arrange.  Often some initial refinements are needed, e.g., auxiliary variable dimension reduction techniques \citep{imcond, immarg} and/or random set stretching \citep{leafliu2012}.  The necessary details for the present setting will be given in Section~\ref{S:main}.

\subsection{Using a family of random sets}
\label{SS:family.S}

It's not necessary to limit oneself to a single random set $\S \sim \prob_\S$ like in the previous subsection.  For additional flexibility, it may be advantageous to consider a family of random sets and, here, I'll consider a family indexed by the parameter space $\Theta$.  That is, consider a fixed auxiliary variable space $\UU$ and a family of supports $\{\SS(\vartheta): \vartheta \in \Theta\}$, each consisting of subsets of $\UU$ satisfying Condition~P1 of Theorem~\ref{thm:valid}.  Then, on each support $\SS(\vartheta)$, define a distribution $\prob_{\S|\vartheta}$ to satisfy Condition~P2.   

Given an association $X=a(\theta, U)$, for $U \sim \prob_U$, where $\prob_U$ is supported on $\UU$, each of the random sets $\S \sim \prob_{\S|\vartheta}$ described above, for $\vartheta \in \Theta$, defines an IM with capacity 
\[ \uPi_x(A \mid \vartheta) = \prob_{\S|\vartheta}\{\Theta_x(\S) \cap A \neq \varnothing\}, \quad A \subseteq \Theta, \quad \vartheta \in \Theta, \]
whose corresponding plausibility contour is 
\[ \pi_x(t \mid \vartheta) = \prob_{\S|\vartheta}\{\Theta_x(\S) \ni t\}, \quad t \in \Theta, \]
which, by definition, is such that each $\uPi_x(\cdot \mid \vartheta)$ satisfies 
\[ \uPi_x(A \mid \vartheta) = \sup_{t \in A} \pi_x(t \mid \vartheta), \quad A \subseteq \Theta. \]
I propose to {\em fuse} the $\vartheta$-specific plausibility measures together at their contours, i.e., 
\[ \pif_x(\vartheta) := \pi_x(\vartheta \mid \vartheta), \quad \vartheta \in \Theta. \]
According to \citet{shafer1976, shafer1987}, \citet[][Sec.~7.8]{lower.previsions.book}, etc., if
\begin{equation}
\label{eq:max}
\sup_{\vartheta \in \Theta} \pif_x(\vartheta) = 1. 
\end{equation}
then this can define a genuine plausibility measure via 
\begin{equation}
\label{eq:fuse}
\uPif_x(A) = \sup_{\vartheta \in A} \pif_x(\vartheta), \quad A \subseteq \Theta. 
\end{equation}
Fused plausibility measures has been used informally---without the adjective ``fused''---in applications of the so-called {\em local conditional IMs} \citep[e.g.,][]{imcond, imvch, imunif, martin.syring.isipta19}, which were designed for valid and efficient inference in some challenging non-regular problems.  

Moreover, validity of the IM corresponding to this fused plausibility function follows immediately from that for the individual $\vartheta$-specific IMs implied by Theorem~\ref{thm:valid}.  

\begin{thm}
\label{thm:valid.fuse}
If the $\vartheta$-specific IM is valid for each $\vartheta \in \Theta$, and \eqref{eq:max} holds, then the IM with fused plausibility measure \eqref{eq:fuse} is valid too.  
\end{thm}

\begin{proof}
The essential observation is that, 
\[ \{\text{$\theta \in A$ and $\uPif_X(A) \leq \alpha$}\} \implies \pi_X(\theta \mid \theta) \leq \alpha. \]
Since the $\vartheta$-specific IM is valid for $\vartheta=\theta$, the right-most event above has $\prob_{X|\theta}$-probability no more than $\alpha$.  From this, validity of the IM with fused plausibility follows:
\[ \sup_{\theta \in A} \prob_{X|\theta}\{\uPif_X(A) \leq \alpha\} \leq \sup_{\theta \in A} \prob_{X|\theta}\{ \pi_X(\theta \mid \theta) \leq \alpha\} \leq \alpha. \qedhere \]
\end{proof}

As long as \eqref{eq:max} holds, the quantity defined in \eqref{eq:fuse} is a legitimate plausibility measure, so its mathematical properties are known.  Furthermore, as the previous theorem established, inference based on the fused IM would be valid, so there is no immediate concerns about its use in a statistical problem.  However, the fusion operation is unfamiliar, so its interpretation is less clear.  But there are insights available in the imprecise probability literature, and I'll give some more details about this in Section~\ref{SS:open.fused}.

\section{Main results}
\label{S:main}

\subsection{Characterization of confidence regions}

Recall the sampling model $X \sim \prob_{X|\theta}$ with unknown parameter $\theta \in \Theta$.  Suppose that the parameter of interest is $\phi = \phi(\theta)$, possibly vector-valued, and let $\Phi = \phi(\Theta)$ be its range.  Let $C_\alpha: \XX \to 2^\Phi$ be the rule that determines, for any specified level $\alpha \in [0,1]$, based on data $X$, a $100(1-\alpha)$\% confidence region $C_\alpha(X)$ for $\phi=\phi(\theta)$.  Its defining property is that $C_\alpha(X)$, a random set as a function of $X \sim \prob_{X|\theta}$, satisfies the coverage probability condition \eqref{eq:coverage}.  I'll also assume that the confidence regions are nested, satisfying \eqref{eq:nested.C}.  

Next, take an association, $X=a(\theta, U)$, with $U \sim \prob_U$, like in \eqref{eq:basic.assoc}, consistent with the posited model $\prob_{X|\theta}$.  Note that $X$ here represents the quantity that the confidence region $C_\alpha$ depends on, which may not be the full data.  Indeed, confidence regions often depend only on a minimal sufficient statistic, in which case $X$ represents that statistic and the association need only describe the sampling distribution of that statistic.  

Technically, the distribution $\prob_U$ of the auxiliary variable $U$ is defined on a measurable space $(\UU, \U)$, and I will take $\U$ to be the Borel $\sigma$-algebra relative to the topology ${\cal T}$ on $\UU$, so that $\U$ contains all the closed subsets of $\UU$.  Given this association, define the collection of subsets $\Theta_x(u) = \{\vartheta: x=a(\vartheta, u)\}$ as in \eqref{eq:basic.focal}, and the new collection 
\begin{equation}
\label{eq:S.alpha}
S_\alpha(\vartheta) = \text{clos}(\{u \in \UU: C_\alpha(a(\vartheta, u)) \ni \phi(\vartheta)\}), \quad (\alpha, \vartheta) \in [0,1] \times \Theta, 
\end{equation}
where $\text{clos}(B)$ denotes the closure of $B \subseteq \UU$ with respect to the topology ${\cal T}$.  Following \citet{shafer1976}, I'll refer to these sets $S_\alpha(\vartheta)$ as {\em focal elements}.  There are certain cases in which the focal elements don't depend on $\vartheta$; see Section~\ref{SSS:fusion}. 

Next, consider the following {\em compatibility condition} that links the given confidence region and the selected association:
\begin{equation}
\label{eq:compatible}
\bigcup_{u \in S_\alpha(\vartheta)} \Theta_x(u) \neq \varnothing \quad \text{for all $(x, \vartheta, \alpha) \in \XX \times \Theta \times [0,1]$}. 
\end{equation}
The role played by the compatibility condition is to ensure that the family of random sets constructed in the proof is such that \eqref{eq:max} holds.  In most problems, it is possible to arrange the association such that $\Theta_x(u)$ is non-empty for all $u$, especially when $X$ represents a minimal sufficient statistic of the same dimension as $\theta$, hence compatibility \eqref{eq:compatible} is automatic.  The only exceptions to compatibility that I'm aware of are cases that involve non-trivial constraints on the parameter space, but I postpone the detailed discussion of this compatibility condition to Section~\ref{SS:conditions}.  

Theorem~\ref{thm:complete} below characterizes confidence regions in terms of a valid (fused) IM.  Remember that the starting point is a family of confidence regions for the interest parameter $\phi=\phi(\theta)$, not necessarily for $\theta$ itself.  There generally is no association that directly links the data with $\phi$ so the challenge is constructing an IM on the full parameter space $\Theta$ in such a way that it marginalizes appropriately to the $\phi=\phi(\theta)$ space. 

\begin{thm}
\label{thm:complete}
Let $C_\alpha$ be a family of confidence regions for $\phi=\phi(\theta)$ that satisfies the coverage property \eqref{eq:coverage} and is nested in the sense of \eqref{eq:nested.C}.  Suppose that the sampling model $\prob_{X|\theta}$ admits an association \eqref{eq:basic.assoc} that is compatible with the family of confidence regions in the sense of \eqref{eq:compatible}.  Then there exists a valid IM for $\theta$ with (fused) plausibility measure $\uPif_x$ such that the marginal plausibility regions for $\phi$, 
\begin{equation}
\label{eq:mpl.region}
C_\alpha^\star(x) = \Bigl\{\varphi \in \Phi: \sup_{\vartheta: \phi(\vartheta) = \varphi} \pif_x(\vartheta) > \alpha \Bigr\} 
\end{equation}
satisfy 
\begin{equation}
\label{eq:subset}
C_\alpha^\star(x) \subseteq C_\alpha(x) \quad \text{for all $(x,\alpha) \in \XX \times [0,1]$}. 
\end{equation}
Equality holds in \eqref{eq:subset} for a particular $\alpha$ if the coverage probability function 
\begin{equation}
\label{eq:cp.function}
\varphi \mapsto \inf_{\theta: \phi(\theta) = \varphi} \prob_{X|\theta}\{C_\alpha(X) \ni \varphi\} 
\end{equation}
is constant equal to $1-\alpha$.  
\end{thm}

\begin{proof}
For the A-step of the IM construction, take any association $X = a(\theta, U)$, $U \sim \prob_U$, consistent with the above sampling distribution $\prob_{X|\theta}$ and compatible with the given confidence region in the sense of \eqref{eq:compatible}.  For the P-step, take a family of random sets $\S \sim \prob_{\S|\vartheta}$, indexed by $\vartheta \in \Theta$, with support 
\[ \SS(\vartheta) = \{S_\alpha(\vartheta): \alpha \in [0,1]\}, \]
closed and nested by the definition of $S_\alpha(\vartheta)$ in \eqref{eq:S.alpha}, and distribution $\prob_{\S|\vartheta}$ satisfying 
\begin{equation}
\label{eq:prob}
\prob_{\S|\vartheta}(\S \subseteq K) = \sup_{\alpha: S_\alpha(\vartheta) \subseteq K} \prob_U\{S_\alpha(\vartheta)\}. 
\end{equation}
That $S_\alpha(\vartheta)$ is $\prob_U$-measurable follows from the fact that $S_\alpha(\vartheta)$ is closed and the $\sigma$-algebra $\U$ contains all closed subsets of $\UU$.  Note, also, that $\Theta_x(\S)$ in \eqref{eq:post.focal} is non-empty with $\prob_{\S|\vartheta}$-probability~1 according to \eqref{eq:compatible}.  For the special singleton assertion $\{\vartheta\}$ about $\theta$, the C-step of the $\vartheta$-specific construction returns the plausibility contour 
\begin{equation}
\label{eq:point}
\pif_x(\vartheta) := \pi_x(\vartheta \mid \vartheta) = \prob_{\S|\vartheta}\{\Theta_x(\S) \ni \vartheta\} = \prob_{\S|\vartheta}\{\S \cap \UU_x(\vartheta) \neq \varnothing\}, 
\end{equation}
where $\UU_x(\vartheta)$ is defined as $\UU_x(\vartheta) = \{u: x=a(\vartheta, u)\}$. 
Note that non-emptiness of $\Theta_x(\S)$ with $\prob_{\S|\vartheta}$-probability 1 for each $\vartheta$ implies \eqref{eq:max}.  Therefore, this determines the fused plausibility contour \eqref{eq:fuse} and the corresponding IM with fused plausibility measure $\uPif_x$ is valid based on Theorem~\ref{thm:valid.fuse}.  To establish the desired connection between the plausibility region of this valid IM and the given confidence region, define the index 
\begin{equation}
\label{eq:alpha}
\alpha(x,\vartheta) = \sup\{\alpha \in [0,1]: S_\alpha(\vartheta) \cap \UU_x(\vartheta) \neq \varnothing\}, 
\end{equation}
so that 
\[ \pif_x(\vartheta) = 1-\prob_{\S|\vartheta}\{\S \subseteq S_{\alpha(x,\vartheta)}(\vartheta)\} \leq \alpha(x,\vartheta), \]
where the inequality is due to \eqref{eq:prob} and the coverage property \eqref{eq:coverage}.  Then the marginal plausibility region $C_\alpha^\star(x)$ defined in \eqref{eq:mpl.region} above satisfies:
\begin{align*}
\varphi \in C_\alpha^\star(x) \iff & \sup_{\vartheta: \phi(\vartheta) = \varphi} \pif_x(\vartheta) > \alpha \\
\iff & \pif_x(\vartheta) > \alpha &\text{for some $\vartheta$ with $\phi(\vartheta) = \varphi$} \\
\implies & \alpha(x,\vartheta) > \alpha & \text{for some $\vartheta$ with $\phi(\vartheta) = \varphi$} \\
\iff & S_\alpha(\vartheta) \cap \UU_x(\vartheta) \neq \varnothing & \text{for some $\vartheta$ with $\phi(\vartheta) = \varphi$} \\
\iff & \varphi \in C_\alpha(x). 
\end{align*}
Therefore, the valid IM constructed above has plausibility regions for $\phi=\phi(\theta)$ that are contained in the given confidence regions, completing the proof of the first claim.  

For the claim about when equality holds in \eqref{eq:subset}, note that the one-sided implication ``$\Longrightarrow$'' in the above display becomes two-sided if $\pif_x(\vartheta) = \alpha(x,\vartheta)$.  This equality holds if the coverage probability function \eqref{eq:cp.function} is constant equal to $1-\alpha$. 
\end{proof}


Theorem~\ref{thm:complete} has a number of interesting and important implications; here are two.  First, compared to the discussion in Section~\ref{SS:isnt.valid} which may have suggested that ``coverage + consonance'' leading to a plausibility measure was an optional embellishment on top of an already existing confidence region, Theorem~\ref{thm:complete} makes clear that the plausibility measure interpretation is an inherent part of the confidence region itself.  Indeed, this result is similar to the so-called {\em complete class theorems}
that often appear in the literature on decision theory \citep[e.g.,][Chap.~8]{berger1985}. That is, for any confidence region that attains the nominal coverage probability as in \eqref{eq:coverage}---and satisfies the other mild regularity conditions---there exists a valid IM whose (fused) plausibility measure returns a plausibility region that is no less efficient.  So, in addition to the benefits of the ``plausibility'' interpretation and the validity guarantees, there is no loss of efficiency in restricting oneself to those confidence regions derived from a valid IM.  Second, Theorem~\ref{thm:complete} goes beyond the simply identifying a plausibility measure that agrees with the given confidence region.  It also says that the plausibility measure in question has the form of those derived in Section~\ref{S:source}, i.e., based on (families of) random sets.  This reveals the fundamental importance of that construction and creates opportunities for developing new and improved methods; see Sections~\ref{re:algorithm} and \ref{S:algorithm}.

\subsection{Characterization of tests}

There is a well-known duality between confidence regions and tests, i.e., reject a null hypothesis if the hypothesized value is not included in the confidence region.  However, this strategy doesn't work especially well in practice for composite hypotheses, since it often would be difficult to determine if the confidence region and the hypothesized values have non-empty intersection.  So the development of hypothesis tests---especially for relevant composite hypotheses---is of fundamental importance.  

Recall that $\Theta_0$ is a specified proper subset of $\Theta$, and the goal is to test the hypothesis $H_0: \theta \in \Theta_0$.  Note that there is no need to introduce the interest parameter $\phi$ in this case, since a hypothesis about $\phi=\phi(\theta)$ is also a hypothesis about $\theta$.  Then a family of tests $\{T_\alpha: \alpha \in [0,1]\}$ is a collection of functions, $T_\alpha: \XX \to \{0,1\}$, with the interpretation that the hypothesis is rejected at level $\alpha$ based on data $X=x$ if and only if $T_\alpha(x)=1$.  It turns out that there is a parallel characterization of tests with provable size control, as in \eqref{eq:size}, and valid IMs whose output is a (fused) plausibility measure, similar to that in Theorem~\ref{thm:complete}. Roughly, under certain conditions, Theorem~\ref{thm:complete.test} below establishes that, for every family of tests satisfying \eqref{eq:size}, there exists a valid IM whose (fused) plausibility measure-based tests control Type~I error and are no less powerful.  

These conditions are analogous to those in the case of confidence regions. First, in addition to the validity condition \eqref{eq:size}, I will require that the test's rejection regions,
\[ R_\alpha = \{x \in \XX: T_\alpha(x) = 1\}, \quad \alpha \in [0,1], \]
be nested in the sense that $R_\alpha \subseteq R_{\alpha'}$ for all $\alpha \leq \alpha'$.  In other words, 
\begin{equation}
\label{eq:nested.test}
\text{if $\alpha \leq \alpha'$, then $T_\alpha(x)=1$ implies $T_{\alpha'}(x)=1$}. 
\end{equation}
Towards the second requirement, like in the context of confidence regions, define the focal elements $S_\alpha(\cdot)$ to support the to-be-defined random set as 
\begin{equation}
\label{eq:S.alpha.test}
S_\alpha(\vartheta) = \text{clos}(\{u \in \UU: T_\alpha(a(\vartheta,u)) = 0\}), \quad (\alpha, \vartheta) \in [0,1] \times \Theta_0. 
\end{equation}
Two points to note: first, that, the collection $\{S_\alpha(\vartheta): \alpha \in [0,1]\}$ of sets is nested according to \eqref{eq:nested.test} for each fixed $\vartheta \in \Theta_0$; second, compared to the confidence region case where I needed focal elements indexed by all $\vartheta \in \Theta$, here I only need them for all $\vartheta \in \Theta_0$.  For the second condition, I will require that a version of compatibility \eqref{eq:compatible} holds for the testing-based focal elements \eqref{eq:S.alpha.test} indexed by $\Theta_0$, i.e., 
\begin{equation}
\label{eq:compatible.test}
\bigcup_{u \in S_\alpha(\vartheta)} \Theta_x(u) \neq \varnothing \quad \text{for all $(x, \vartheta, \alpha) \in \XX \times \Theta_0 \times [0,1]$}. 
\end{equation}
As before, the purpose of this compatibility condition is to ensure that the family of random sets constructed in the proof are such that \eqref{eq:max} holds.  

\begin{thm}
\label{thm:complete.test}
For an unknown parameter $\theta \in \Theta$, and a proper subset $\Theta_0 \subset \Theta$, let $\{T_\alpha: \alpha \in [0,1]\}$ be a family of tests of the hypothesis $H_0: \theta \in \Theta_0$ that satisfies the size and nestedness conditions in \eqref{eq:size} and \eqref{eq:nested.test}, respectively.  Suppose that the sampling model $\prob_{X|\theta}$ admits an association \eqref{eq:basic.assoc} that is compatible with the family of tests in the sense that the compatibility condition \eqref{eq:compatible.test} holds for the focal elements in \eqref{eq:S.alpha.test}.  Then there exists a valid IM for $\theta$ such that, if $T_\alpha^\star$ is the corresponding (fused) plausibility measure-based test 
\[ T_\alpha^\star(x) = 1\bigl\{\uPif_x(\Theta_0) \leq \alpha\bigr\}, \quad (x,\alpha) \in \XX \times [0,1], \]
where $\uPif_x$ is defined in \eqref{eq:fuse}, then 
\begin{equation}
\label{eq:aggressive}
T_\alpha(x) \leq T_\alpha^\star(x) \quad \text{for all $(x,\alpha) \in \XX \times [0,1]$}. 
\end{equation}
Equality holds in \eqref{eq:aggressive}, for a particular $\alpha$, if the size of $T_\alpha$ in \eqref{eq:size} is exactly $\alpha$.
\end{thm}

\begin{proof}
Since the test functions only take values 0 and 1, it suffices to show that $T_\alpha^\star(x) = 0$ implies $T_\alpha(x)=0$.  Towards this, I proceed in a way analogous to that in the proof of Theorem~\ref{thm:complete}.  I define a $\vartheta$-dependent random set $\S \sim \prob_{\S|\vartheta}$ as follows.  First, I set the support of $\prob_{\S|\vartheta}$ to be the collection
\[ \SS(\vartheta) = \{S_\alpha(\vartheta): \alpha \in [0,1]\}, \quad \vartheta \in \Theta, \]
where $S_\alpha(\vartheta)$ is as defined in \eqref{eq:S.alpha.test}.  Next, just like in \eqref{eq:prob}, I take the measure $\prob_{\S|\vartheta}$ to be
\[ \prob_{\S|\vartheta}(\S \subseteq K) = \sup_{\alpha: S_\alpha(\vartheta) \subseteq K} \prob_U\{S_\alpha(\vartheta)\}, \quad K \subseteq \UU. \]
Again, $S_\alpha(\vartheta)$ is $\prob_U$-measurable by assumption, and the image $\Theta_x(\S)$ on the parameter space is non-empty with $\prob_{\S|\vartheta}$-probability~1 for all $x$ and $\vartheta$ by the compatibility assumption.  This determines a valid IM with (fused) plausibility measure $\uPif_x$ and corresponding contour function $\pif_x$ just like in the proof of Theorem~\ref{thm:complete}.  

To establish the desired connection between this IM's test of $H_0$ and the given test, first define the index $\alpha(x,\vartheta)$ as in \eqref{eq:alpha}.  Then it follows from the above construction---in particular, \eqref{eq:point}---and the size constraint \eqref{eq:size} on the given family of tests, that 
\begin{align*}
\pif_x(\vartheta) & = \prob_{\S|\vartheta}\{\S \cap \UU_x(\vartheta) \neq \varnothing\} \\
& = 1-\prob_{\S|\vartheta}\{\S \subseteq S_{\alpha(x,\vartheta)}(\vartheta)\} \\
& \leq \alpha(x,\vartheta), 
\end{align*}
where, recall, $\UU_x(\vartheta)=\{u: x=a(\vartheta, u)\}$.  For the two tests $T_\alpha$ and $T_\alpha^\star$, the following relationship holds:
\begin{align*}
T_\alpha^\star(x) = 0 \iff & \uPif_x(\Theta_0) > \alpha \\
\iff & \sup_{\vartheta \in \Theta_0} \pif_x(\vartheta) > \alpha \\
\iff &\pif_x(\vartheta) > \alpha &\text{for some $\vartheta \in \Theta_0$} \\
\implies & \alpha(x,\vartheta) > \alpha & \text{for some $\vartheta \in \Theta_0$} \\
\iff & S_\alpha(\vartheta) \cap \UU_x(\vartheta) \neq \varnothing & \text{for some $\vartheta \in \Theta_0$} \\
\iff & T_\alpha(x) = 0. 
\end{align*}
This proves the first claim of the theorem.  In words, the plausibility measure-based test constructed above has a corresponding size-$\alpha$ test that rejects $H_0$ at least for all the same $x$ that $T_\alpha$ does, and perhaps even for some $x$ that $T_\alpha$ does not.  Therefore, between the two size-$\alpha$ tests, $T_\alpha^\star$ can be no less powerful than $T_\alpha$.  

The two tests are equivalent if the one-sided arrow ``$\Longrightarrow$'' above could be made two-sided.  This holds if and only if the original test's size in \eqref{eq:size} is exactly $\alpha$.  
\end{proof}

The theorem above is stated in terms of the testing decision rules, i.e., $T_\alpha$ and $T_\alpha^\star$, but more can be said.  Indeed, since both tests have a p-value, and since \eqref{eq:aggressive} holds for all $\alpha$ and all $x$, it must be that the p-value for the IM-based test is no smaller than that of the given test, uniformly in $x$. Compare this to \citet{impval} who only give sufficient conditions for when the two p-values are identical.

\section{Technical remarks}
\label{S:remarks}

\subsection{On the compatibility condition \eqref{eq:compatible}}
\label{SS:conditions}


The compatibility condition \eqref{eq:compatible} in Theorem~\ref{thm:complete} is technical and not so transparent, so it will help for me to provide some further explanation.  First, the compatibility condition is trivial whenever $\Theta_x(u)$ is non-empty for every pair $(x,u)$.  This would be the case for regular problems where $X$ represents a minimal sufficient statistic of the same dimension as $\theta$ and there are no non-trivial constraints on the parameter space.  If the confidence region $C_\alpha$ is a function of only the minimal sufficient statistic, then \eqref{eq:compatible} follows directly.  Many problems are of this type.  

Second, the aforementioned regularity is not necessary for compatibility.  Suppose that $Y_1,\ldots,Y_n$ are iid $\unif(\theta,\theta+1)$, and define the minimal sufficient statistic $X=(X_1,X_2)$, the sample minimum and maximum, respectively.  Since $X$ is two-dimensional but $\theta$ is a scalar, the model is considered to be ``non-regular.''  Here I will consider a Bayesian approach and, since this is a location problem, I assign a flat prior to $\theta$; then the posterior is $\unif(X_2-1, X_1)$, and the equi-tailed $100(1-\alpha)$\% credible interval is 
\[ C_\alpha(x) = \bigl[x_1 - (1-d_x)(1-\tfrac{\alpha}{2}), \, x_1 - (1-d_x) \tfrac{\alpha}{2} \bigr], \quad d_x = x_2 - x_1, \]
and it is nested and satisfies the coverage condition \eqref{eq:coverage} with equality.  The natural association is $X=\theta + U$, where $U=(U_1,U_2)$ denotes the minimum and maximum of an iid sample of size $n$ from $\unif(0,1)$.  With this choice, 
\begin{align*}
S_\alpha(\vartheta) & = \{(u_1,u_2): [u_1 + \vartheta - (1-d_u)(1-\tfrac{\alpha}{2}), \, u_1 + \vartheta - (1-d_u)\tfrac{\alpha}{2}] \ni \vartheta \} \\
& = \{(u_1, u_2): \tfrac{\alpha}{2-\alpha} (1-u_2) \leq u_1 \leq \tfrac{2-\alpha}{\alpha} (1-u_2) \}.
\end{align*}
Note that any $(u_1,u_2)$ with $u_1 \leq u_2$ and $u_1 = 1-u_2$ belongs to $S_\alpha(\vartheta)$ for all $\alpha$ and for all $\vartheta$.  In particular, if $\hat\theta(x) = \frac12\{x_1 + (x_2-1)\}$ and $u(x)=x - \hat\theta(x) 1_2$, then $u(x) \in S_\alpha(\vartheta)$ and $\Theta_x(u(x)) = \{\hat\theta(x)\} \neq \varnothing$, which implies \eqref{eq:compatible}.  

However, it's not necessary that $\Theta_x(u)$ be non-empty for each $(x,u)$.  For example, consider an iid normal mean model, with association $X=\theta 1_n + U$, where $X=(X_1,\ldots,X_n)$, $U=(U_1,\ldots,U_n)$, $\theta \in \RR$, $1_n$ is a $n$-vector of unity, and $\phi = \theta$.  Then it is easy to check that $\Theta_x(u)$ is non-empty only for pairs of $n$-vectors $(x,u)$ that differ by a constant shift, which is a set of Lebesgue measure zero.  However, the textbook confidence interval for $\theta$ in this case is $C_\alpha(x) = \xbar \pm z_\alpha^\star n^{-1/2}$, so 
\[ S_\alpha(\vartheta) = \{u: |\ubar| \leq z_\alpha^\star n^{-1/2}\}, \quad \forall \; \vartheta \in \Theta. \]
Then the vector $u = x - \xbar 1_n$ belongs to $S_\alpha(\vartheta)$ for each $(\vartheta, \alpha)$ and, therefore, for this choice of $u$, $\Theta_x(u) = \{\xbar\} \neq \varnothing$, hence \eqref{eq:compatible} holds. 

Based on my experience, the only cases where the compatibility condition \eqref{eq:compatible} may fail is in cases where there are non-trivial parameter constraints, either by design---like in the Poisson-plus-background examples arising in high-energy physics \citep[e.g.,][Sec.~2.1.2]{mandelkern2002}---or by some unique feature of the model formulation.  My example here is of the latter type.  Let $X$ be an observation from the non-central chi-square distribution with known degrees of freedom $d$ but unknown non-centrality parameter $\theta$.  The natural association in this setting is $X = F_\theta^{-1}(U)$, where $U \sim \unif(0,1)$, where $F_\theta$ is the non-central chi-square distribution function.  What creates a problem in this example is the fact that the basic focal elements, $\Theta_x(u) = \{\vartheta: F_\vartheta(x)=u\}$, are empty for sets of $u$ with positive probability.  In particular, if $u > F_0(x)$, where $F_0$ is the central chi-square distribution function, then $\Theta_x(u) = \varnothing$.  Next, consider the basic upper confidence bound 
\[ C_\alpha(x) = \{\vartheta: F_\vartheta(x) \geq \alpha\}, \]
which has exact coverage probability $1-\alpha$.  It is also easy to check that  
\[ S_\alpha(\vartheta) = \{u: C_\alpha(F_\vartheta^{-1}(u)) \ni \vartheta\} = [\alpha,1]. \]
Therefore, for any pair $(x,\alpha)$ such that $F_0(x) \leq \alpha$, it follows that $\bigcup_{u \in [\alpha,1]} \Theta_x(u) = \varnothing$, hence the compatibility condition \eqref{eq:compatible} fails.  

To summarize, the compatibility condition seems to be rather weak, and the only examples I'm aware of where it fails are easily characterized---as involving non-trivial parameter constraints---and arguably obscure.  But a rigorous proof that compatibility holds for one class of examples but not for another presently escapes me.

\subsection{Is fusion necessary?}
\label{SSS:fusion}

Fusion is not absolutely necessary in the sense that there are problems where the focal element's dependence on the parameter disappears and just a single random set would do.  To see this, consider a case where both the statistical model and the confidence region have the following transformation structure.  Let $\G$ be a group of transformations, $g: \XX \to \XX$, with associated group $\Gbar$, such that $X \sim \prob_\theta$ implies $gX \sim \prob_{\bar g \theta}$; here I'll assume that $\Gbar$ acts transitively on $\Theta$.  (Two quick comments about the notation: first, as is customary in this context, I write $gx$ instead of $g(x)$; second, for the present discussion only, I've dropped the ``$X|$'' part of the subscript on $\prob_{X|\theta}$.)  Next, define another associated group $\Gtilde$ such that $\phi(\bar g\theta) = \tilde g \phi(\theta)$.  Then a confidence region $C_\alpha$ for $\phi$ is {\em invariant} if $C_\alpha(gx) = \tilde g C_\alpha(x)$ for all pairs $(x,g) \in \XX \times \G$.  \citet{arnold1984} gives further details on this setup, along with some examples.  For the present context, the key point is that the model can be described by an arbitrary ``baseline'' value $\theta_0$ of $\theta$ and a mapping $g_\theta$ that converts $U \sim \prob_{\theta_0}$ to $X=g_\theta U$ with $X \sim \prob_\theta$.  Therefore, the association \eqref{eq:basic.assoc} looks like $X=g_\theta U$, with $U \sim \prob_{\theta_0}$.  Moreover, if the transformation structure is in terms of a minimal sufficient statistic $X$, as in Arnold, then, e.g., $\G$ acting transitively on $\XX$ implies $\Theta_x(u) \neq \varnothing$ for all $(x,u)$, hence \eqref{eq:compatible}.  To the question of whether fusion is necessary, note that the focal elements $S_\alpha(\vartheta)$ in \eqref{eq:S.alpha}  are free of $\vartheta$, i.e., 
\begin{align*}
S_\alpha(\vartheta) & = \{u \in \XX: C_\alpha(g_\vartheta u) \ni \phi(\vartheta)\} \\
& = \{u \in \XX: \tilde g_\vartheta C_\alpha(u) \ni \tilde g_\vartheta \phi(\theta_0)\} \\
& = \{u \in \XX: C_\alpha(u) \ni \phi(\theta_0)\}. 
\end{align*}
If the random set's focal elements do not depend on $\vartheta$, then this is a ``one random set'' case as described in Section~\ref{SS:one.S}, hence no fusion required.  

However, without the above group transformation structure, it's generally not possible to remove parameter-dependence from the focal elements.  To see this, consider the following alternative formulation with parameter-free focal elements 
\[ S_\alpha = \bigcap_{\vartheta \in \Theta} S_\alpha(\vartheta) = \{u: C(a(\vartheta, u)) \ni \phi(\vartheta) \text{ for all $\vartheta \in \Theta$}\}. \]
Since these sets are free of $\vartheta$, the single random set formulation could be applied.  Recall that a critical piece of the above proof was the fact that $\prob_U\{S_\alpha(\vartheta)\} \geq 1-\alpha$ for each $\vartheta$, a consequence of \eqref{eq:coverage}.  Considering the new $\vartheta$-free focal elements defined above, since $S_\alpha$ is generally much smaller than any individual $S_\alpha(\vartheta)$, there is reason to doubt that the analogous relation ``$\prob_U(S_\alpha) \geq 1-\alpha$'' holds.  So, unless there is some special structure, like the group transformation structure described above, which implies that $S_\alpha(\vartheta) \equiv S_\alpha$, then the step that fuses together a family parameter-dependent plausibility functions apparently can't be avoided.  

I'll end with two comments related to fusion in the context of hypothesis tests.  First, compared to the case of confidence regions, the present situation is potentially different because the above construction only requires focal elements indexed by points $\vartheta$ in $\Theta_0$.  So, in the extreme---but sometimes practically relevant---case of a point null hypothesis, fusion is not necessary.  Since a confidence region can be interpreted as a collection of point null tests corresponding to every point in the parameter space, it is comforting that the technical requirements are much weaker in cases where only a single point null test is desired.  But the real utility of formal hypothesis tests compared to confidence regions is their ability to handle non-trivial composite hypotheses.

Second, the construction in \citet{impval}, which establishes a  connection between IM plausibilities and classical p-values, does not involve the fusion step.  How is this possible?  To see the connection most clearly, I'll express the test $T_\alpha$ in terms of a real-valued test-statistic $\tau$, i.e., so that $T_\alpha(x) = 1\{\tau(x) > t_\alpha\}$, where $t_\alpha$ is an appropriate critical value, then the focal elements take the form 
\[ S_\alpha(\vartheta) = \text{clos}(\{u \in \UU: \tau(a(\vartheta,u)) < t_\alpha\}), \quad \vartheta \in \Theta_0. \]
As above, an idea to eliminate the $\vartheta$-dependence is to take an intersection,
\[ \bigcap_{\vartheta \in \Theta_0} S_\alpha(\vartheta) = \Bigl\{u \in \UU: \sup_{\vartheta \in \Theta_0} \tau(a(\vartheta, u)) \leq t_\alpha \Bigr\}. \]
The reader may recognize the right-hand side as exactly the focal elements used in the \citet{impval} construction (modulo the equivalent choice to index by $t=t_\alpha$ instead of $\alpha$).  But this intersection operation shrinks the focal elements considerably, which puts extra demands on other aspects of the problem.  For example, it may require that the model have a group transformation structure and that the testing problem be suitably invariant.  \citet{impval} impose the following condition on $\tau$,  
\[ \prob_U\Bigl\{ \sup_{\vartheta \in \Theta_0} \tau\bigl(a(\vartheta,U)\bigr) \leq t_\alpha \Bigr\} = \inf_{\vartheta \in \Theta_0} \prob_U\bigl\{ \tau\bigl(a(\vartheta,U)\bigr) \leq t_\alpha \bigr\}, \quad \text{for all $\alpha$}, \]
which must be roughly as restrictive as the aforementioned group invariance assumption.  Therefore, the result in Theorem~\ref{thm:complete.test}, which allows for fusion, is significantly more general then the result in \citet{impval}.

\subsection{From a constructive proof to a practical algorithm}
\label{re:algorithm}

In most theoretical investigations in statistics, the proof strategy is practically irrelevant.  Here, however, the fact that the proofs of Theorems~\ref{thm:complete} and \ref{thm:complete.test} are {\em constructive} has some important practical consequences.  Specifically, the construction of a valid IM---via random sets and a fusion operation---in the two proofs effectively defines an {\em algorithm} one can use to develop new IM-based methods.  That is, take a given family of procedures, define a family of $\vartheta$-dependent random sets with focal elements as in \eqref{eq:S.alpha}, then its distribution as in \eqref{eq:prob}, and finally return the fused plausibility measure as in \eqref{eq:fuse}, from which a potentially new procedure can be derived, which is provably valid and no-less-efficient then the original procedure.  Several examples where this {\em IM algorithm} is carried out are presented in Sections~\ref{S:examples} and \ref{S:more} below, with some more specific details in Section~\ref{S:algorithm} about the algorithm's implementation and where to find ``default'' procedures satisfying \eqref{eq:size} and/or \eqref{eq:coverage} that can be used as input to this IM algorithm.

\subsection{Direct vs.~indirect IM construction}

In Section~\ref{S:source}, I described how to construct a valid IM directly without making use of an existing test or confidence region, whereas the main results in Section~\ref{S:main} focus on a construction driven by a given test or confidence region.  From a practical point of view, what's the benefit of one approach compared to the other?  

It deserves mention that, in cases where the given confidence region is only for a feature $\phi$ of $\theta$, then the IM for $\theta$ determined by Theorem~\ref{thm:complete} would be focused primarily on that feature.  That is, while inference would be valid for any assertion about $\theta$, it would generally be inefficient except for assertions specifically relevant to $\phi$.  In that case, unless I was sure that the only assertions I cared about where those concerning $\phi$, then I might prefer the direct IM construction because it would give me the flexibility to answer other questions with some efficiency.  If the given confidence region is for the full parameter $\theta$, then the construction coming from Theorem~\ref{thm:complete} is more attractive.  First, it's free in the sense that I don't have to think of something myself; second, it's sure to be at least as efficient as the confidence regions I started with.  At the very least, the characterization result in Theorem~\ref{thm:complete} should provide some insights on what are the ``good'' random sets to use in the direct construction.

\section{Illustrations}
\label{S:examples}

Here I present three relatively simple examples where a standard confidence region is available that achieves the coverage property \eqref{eq:coverage} and can be used as input to the ``IM algorithm.''  Section~\ref{S:more} considers two non-trivial examples, in the context of hypothesis testing, based on new ideas developed below in Section~\ref{S:algorithm}.  

\subsection{Binomial problem}
\label{SS:binomial}

Let $X \sim \bin(n,\theta)$, where the size $n$ is known but the success probability $\theta \in [0,1]$ is unknown.  A standard confidence interval for $\theta$, satisfying the coverage probability condition \eqref{eq:coverage}, is the Clopper--Pearson interval
\[ C_\alpha(x) = \{\vartheta: F_\vartheta(x) \geq \tfrac{\alpha}{2}, \, 1-F_\vartheta(x-1) \geq \tfrac{\alpha}{2}\}, \quad \alpha \in (0,1), \]
where $F_\theta$ is the $\bin(n,\theta)$ distribution function.  It is well-known \citep[e.g.,][Figure~11]{bcd2001} that the coverage probability of $C_\alpha$ varies wildly as a function of $\theta$ so it does not have the nice ``uniform exactness'' properties that pivot-based intervals enjoy.  Therefore, the IM algorithm coming out of the proof of Theorem~\ref{thm:complete} may lead to some improvements over the Clopper--Pearson interval.  

To start, I will take the association in \eqref{eq:basic.assoc} as 
\[ F_\theta(X-1) < U \leq F_\theta(X), \quad U \sim \unif(0,1). \]
This is a standard formula for simulating a binomial variate, and I write the solution $X$ as $F_\theta^{-1}(U)$, which can be computed via the {\tt qbinom} function in R.  Then the support set $S_\alpha(\vartheta)$ is given by 
\[ S_\alpha(\vartheta) = \{u: F_\vartheta(F_\vartheta^{-1}(u)) \geq \tfrac{\alpha}{2}, \, 1-F_\vartheta(F_\vartheta^{-1}(u)-1) \geq \tfrac{\alpha}{2}\}. \]
The set $\UU_x(\vartheta) = \{u: F_\vartheta(x-1) < u \leq F_\vartheta(x)\}$ is an interval, and I need the supremum of all $\alpha$ such that $S_\alpha(\vartheta) \cap \UU_x(\vartheta) \neq \varnothing$.  This value is the index $\alpha(x,\vartheta)$ in \eqref{eq:alpha}.  A unique feature of this and other discrete data problems is that the support sets $S_\alpha(\vartheta)$ may not be empty as $\alpha \to 1$, so it is possible for $S_\alpha(\vartheta) \cap \UU_x(\vartheta)$ to be non-empty for all $\alpha$, in which case $\alpha(x,\vartheta) = 1$.  In the present case, this index is given by 
\[ \alpha(x,\vartheta) = \begin{cases}
2\{1 - F_\vartheta(x-1)\}, & \text{if $F_\vartheta(x-1) \geq \frac12$} \\
2 F_\vartheta(x), & \text{if $F_\vartheta(x) \leq \frac12$} \\
1, & \text{if $F_\vartheta(x-1) < \frac12 < F_\vartheta(x)$}.
\end{cases} \]
Then the IM's fused plausibility contour is 
\[ \pif_x(\vartheta) = \begin{cases}
1, & \text{if $F_\vartheta(x-1) < \frac12 < F_\vartheta(x)$} \\
1 - g(\vartheta,x) , & \text{otherwise}.
\end{cases} \]
where
\[ g(\vartheta, x) = \prob_U\{F_\vartheta(F_\vartheta^{-1}(U)) > \tfrac{\alpha(x,\vartheta)}{2}, \, F_\vartheta(F_\vartheta^{-1}(U)-1) < 1-\tfrac{\alpha(x,\vartheta)}{2}\}, \]
which can be readily computed numerically, with or without Monte Carlo; see below. 

As an illustration, consider an experiment with $n=25$ trials and $X=17$ observed successes.  Figure~\ref{fig:binom} shows a plot of the plausibility contour extracted directly from the Clopper--Pearson confidence interval, which is just $\alpha(x,\theta)$, and that for the IM described above.  Note that the latter is no wider than that from Clopper--Pearson.  In fact, the fused plausibility contour is exactly the ``acceptability function'' developed in \citet{blaker2000}, hence, the corresponding plausibility interval is the same as Blaker's interval, which has been previously shown to be a significant improvement over the classical Clopper--Pearson interval.  This gives a concrete demonstration of the  improvement resulting from the IM construction.  Also, one can use available software packages \citep[e.g.,][]{blaker.R} for efficient IM-based inference in this binomial application.  See, also, \citet{balch.binomial} for a more recent improvement, also based on imprecise probabilistic ideas.

\begin{figure}[t]
\begin{center}
\scalebox{0.70}{\includegraphics{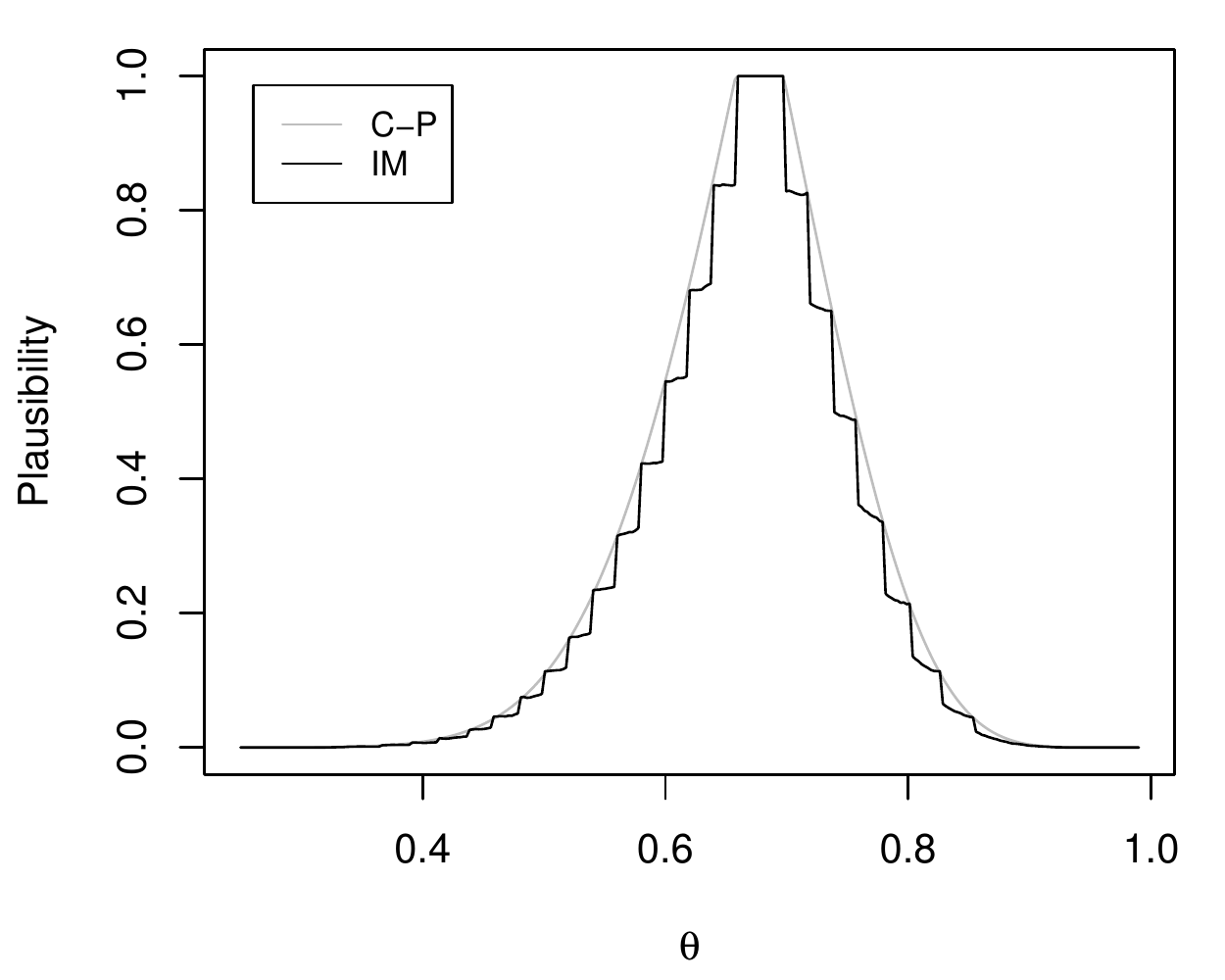}}
\end{center}
\caption{Plot of the IM and the Clopper--Pearson (C--P) plausibility contours for the binomial problem, with $n=25$ and $X=17$.}
\label{fig:binom}
\end{figure}

\subsection{Behrens--Fisher problem}
\label{SS:bf}

Consider two independent data sets $Y_{11},\ldots,Y_{1n_1}$ and $Y_{21},\ldots,Y_{2n_2}$ from $\nm(\mu_1, \sigma_1^2)$ and $\nm(\mu_2, \sigma_2^2)$, respectively, where $\theta=(\mu_1,\mu_2,\sigma_1^2, \sigma_2^2)$ is unknown, but the goal is inference on $\phi(\theta) = \mu_1-\mu_2$.  Let $X=(M_1, M_2, V_1, V_2)$ denote the minimal sufficient statistic, consisting of the two sample means and two sample variances, respectively.  Then the famous Hsu--Scheff\'e interval \citep{hsu1938, scheffe1970} is given by 
\[ C_\alpha(x) = (m_1 - m_2) \pm t_{\alpha,n}^\star \, f(v_1, v_2), \]
where $f(v_1, v_2) = (v_1/n_1 + v_2/n_2)^{1/2}$ and $t_{\alpha,n}^\star$ is the $1-\frac{\alpha}{2}$ quantile of a Student-t distribution with $\min\{n_1,n_2\} - 1$ degrees of freedom.  It follows from results in, e.g.,  \citet{mickey.brown.1966} that the above interval achieves the coverage probability condition \eqref{eq:coverage}, but could be somewhat conservative for small $n$ and/or certain $\theta$ configurations.  Here I derive an IM counterpart for $C_\alpha$ following the proof of Theorem~\ref{thm:complete}.  

For the association \eqref{eq:basic.assoc}, I will take 
\[ D = \phi + f(\sigma_1^2, \sigma_2^2) \, U_1, \quad V_k = \sigma_k^2 U_{2k}, \quad k=1,2, \]
where $D=M_1-M_2$ and the auxiliary variable $U=(U_1,U_{21},U_{22})$ consists of independent components with $U_1 \sim \nm(0,1)$ and $U_{2k} \sim \chisq(n_k-1) / (n_k - 1)$, $k=1,2$.  Since 
\[ \frac{D - \phi}{f(V_1,V_2)} = \frac{f(\sigma_1^2, \sigma_2^2)}{f(\sigma_1^2 U_{21}, \sigma_2^2 U_{22})} \, U_1, \]
the support sets $S_\alpha(\vartheta)$ can be written as  
\[ S_\alpha(\vartheta) = \Bigl\{ u=(u_1, u_{21}, u_{22}): \frac{|u_1|}{\{\lambda_\vartheta u_{21} + (1-\lambda_\vartheta) u_{22}\}^{1/2}} \leq t_{\alpha, n}^\star \Bigr\}, \]
where $\lambda_\vartheta = \{1 + (n_1 \sigma_1^2)/(n_2 \sigma_2^2)\}^{-1}$, which takes values in $[0,1]$ as a function of $\vartheta$.  Note that the set $S_\alpha(\vartheta)$ depends on $\vartheta$ only through $\lambda_\vartheta$, not on $\phi$ or on the specific values of $\sigma_1^2$ and $\sigma_2^2$.  Then the index $\alpha(x,\vartheta)$ in \eqref{eq:alpha} is given by 
\[ \alpha(x,\vartheta) = 2 \bigl| F_n\bigl( t(x,\vartheta) \bigr) - 1 \bigr|, \]
where $t(x,\vartheta) = (d - \varphi)/f(v_1,v_2)$ and $F_n$ is the Student-t distribution function with $\min\{n_1,n_2\}-1$ degrees of freedom, which actually only depends on $\varphi$.  Then the fused plausibility measure has contour $\pif_x(\vartheta) = \prob_U\{S_{\alpha(x,\vartheta)}(\vartheta)\}$, which can be readily evaluated via Monte Carlo.  Then the corresponding marginal plausibility contour for $\phi$ can be obtained by optimization over $\lambda$.  

For illustration, consider the example in \citet[][p.~83]{lehmann1975} on travel times for two different routes; summary statistics are:
\begin{align*}
n_1 & = 5 & m_1 & = 7.580 & v_1 & = 2.237 \\
n_2 & = 11 & m_2 & = 6.136 & v_2 & = 0.073.
\end{align*}
The goal is to compare the mean travel times for two different routes.  Figure~\ref{fig:bf} shows a plot of the plausibility contour $\alpha(x,\cdot)$ for $\phi$ extracted directly from the Hsu--Scheff\'e confidence interval, the $\lambda$-specific plausibility contours above for a range of $\lambda$, and the corresponding marginal plausibility contour based on optimization over $\lambda$.  As expected, the $\lambda$-specific plausibilities (gray lines) are individually more efficient than that of Hsu--Scheff\'e.  However, marginalizing over $\lambda$ via optimization widens the plausibility contours to agree exactly with Hsu--Scheff\'e, as predicted by Theorem~\ref{thm:complete}.  

\begin{figure}[t]
\begin{center}
\scalebox{0.70}{\includegraphics{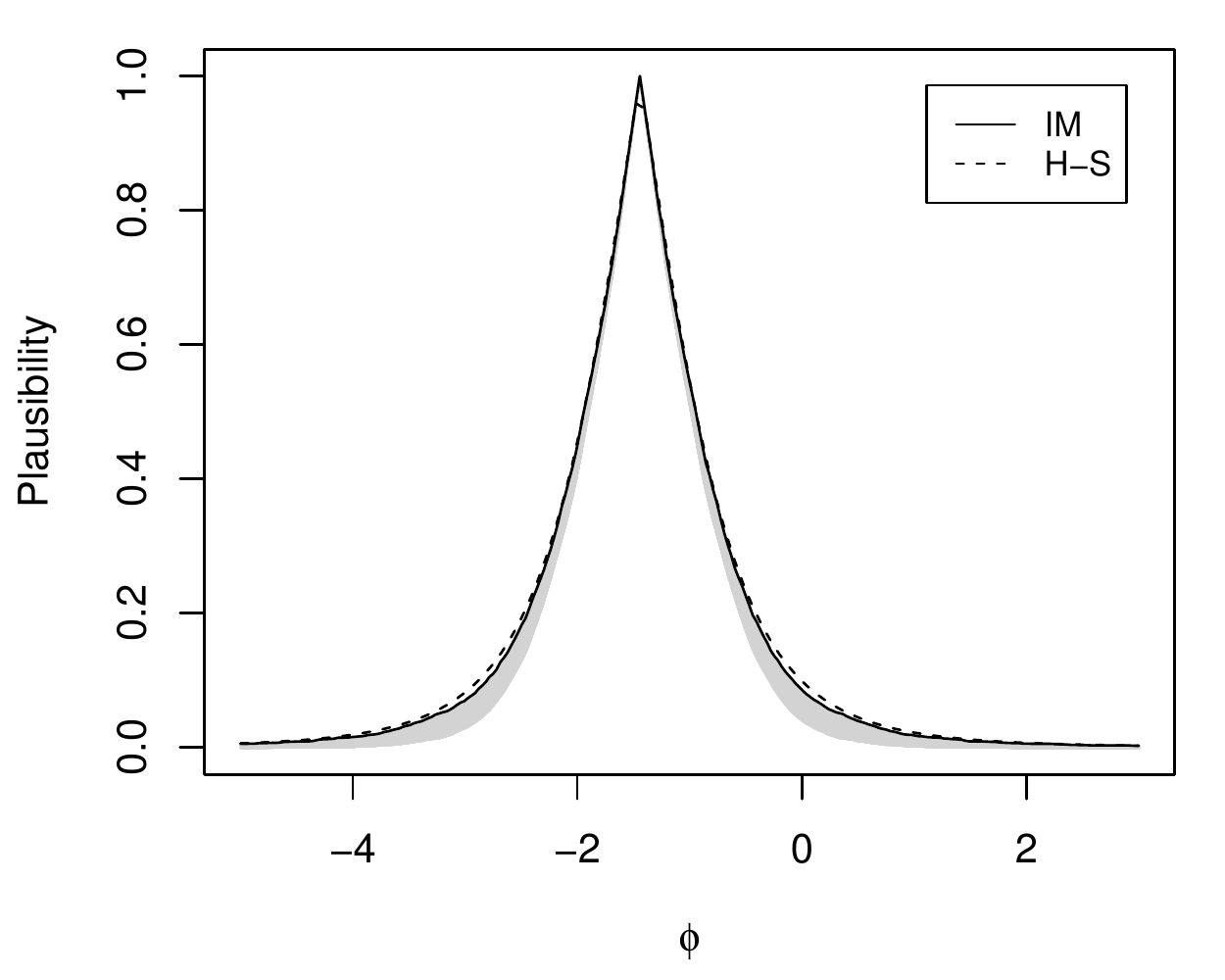}}
\end{center}
\caption{IM and Hsu--Scheff\'e (H--S) plausibility contours for the travel times illustration of the Behrens--Fisher problem. Differences between the solid and dashed lines are the result of optimizing over only a finite range of $\lambda$.}
\label{fig:bf}
\end{figure}

\subsection{Nonparametric one-sample problem}
\label{SS:one.sample}

Although the paper's developments so far focus primarily on finite-dimensional or parametric models, there is no reason that they can't be applied to infinite-dimensional or nonparametric problems just the same.  As a simple illustration, consider real-valued data $X=(X_1,\ldots,X_n)$ iid with distribution function $F$, where the goal is inference on the ``parameter'' $F$.  A standard nonparametric confidence region is based on the Dvoretsky--Kiefer--Wolfowitz inequality \citep[DKW,][]{dkw1956} and corresponds to a ball with respect to the sup-norm $\|\cdot\|_\infty$ on $\RR$, i.e., 
\[ C_\alpha(x) = \{F: \|\hat F_x - F\|_\infty \leq \delta_{n,\alpha}\}, \]
where $\hat F_x$ denotes the empirical distribution function based on a sample $x=(x_1,\ldots,x_n)$, and $\delta_{n,\alpha} = (2n)^{-1/2}\{\log(2/\alpha)\}^{1/2}$.  Equivalently, $C_\alpha(x)$ can be viewed as a confidence band, giving pointwise lower and upper confidence limits; see Figure~\ref{fig:np} below.  Since the DKW inequality upon which the coverage probability results for $C_\alpha$ are based holds for all $F$, there will be some $F$ at which the coverage is conservative, it's possible that the procedure derived from the IM algorithm will be more efficient.  

To keep things simple here, I'll write $F^{-1}$ for the inverse of a distribution function $F$, and ignore the possible non-uniqueness due to $F$ flattening out on some interval.  Then a natural association is 
\[ X_i = F^{-1}(U_i), \quad i=1,\ldots,n, \]
where $U=(U_1,\ldots,U_n)$ are iid $\unif(0,1)$.  Writing this in vectorized form, i.e., $X=F^{-1}(U)$, it follows that 
\[ S_\alpha(F) = \{u: C_\alpha(F^{-1}(u)) \ni F\} = \{u: \|\hat F_{F^{-1}(u)} - F\|_\infty \leq \delta_{n,\alpha}\}, \]
and it is easy to confirm that $S_\alpha(F) \equiv S_\alpha$ does not depend on $F$, hence there is no need for fusion.  Furthermore, the index $\alpha(x,F)$ defined in the proof of Theorem~\ref{thm:complete} is 
\[ \alpha(x,F) = \min\bigl\{ 1, 2 e^{-2n \|\hat F_x - F\|_\infty^2} \bigr\}. \]
Therefore, the plausibility contour is $\pi_x(F) = 1-\prob_U\{S_{\alpha(x,F)}\}$, which can readily be evaluated via Monte Carlo for any fixed $F$.  

For an illustration, consider the nerve data set from \citet{cox.lewis.book}, analyzed in \citet[][Example~2.1]{wasserman2006book} on waiting times between pulses along a nerve fiber.  The original data has 799 observations but, to make differences between various methods easier to see for this illustration, I'm working with a randomly chosen subsample of size $n=100$.  A plot of the empirical distribution function along with the lower and upper 95\% confidence bands based on the DKW inequality is shown in Figure~\ref{fig:np}.  While it's computationally a challenge to draw a new pair of plausibility bands corresponding to the output of the IM algorithm, it's easy to confirm that the bands are slightly narrower, as predicted by Theorem~\ref{thm:complete}.  Take the candidate $F$ equal to the lower bound in Figure~\ref{fig:np}, so that $\alpha(x,F) \approx 0.05$.  However, it turns out that $\pi_x(F) \approx 0.045$.  Since the 95\% lower confidence bound based on the DKW inequality is not included in the IM algorithm's 95\% plausibility region, it follows that the latter is narrower and, therefore, slightly more efficient.  With the full data set, however, the DKW region and that based on the IM algorithm are roughly the same.  

\begin{figure}[t]
\begin{center}
\subfigure[Empirical distribution and DKW bands]{\scalebox{0.6}{\includegraphics{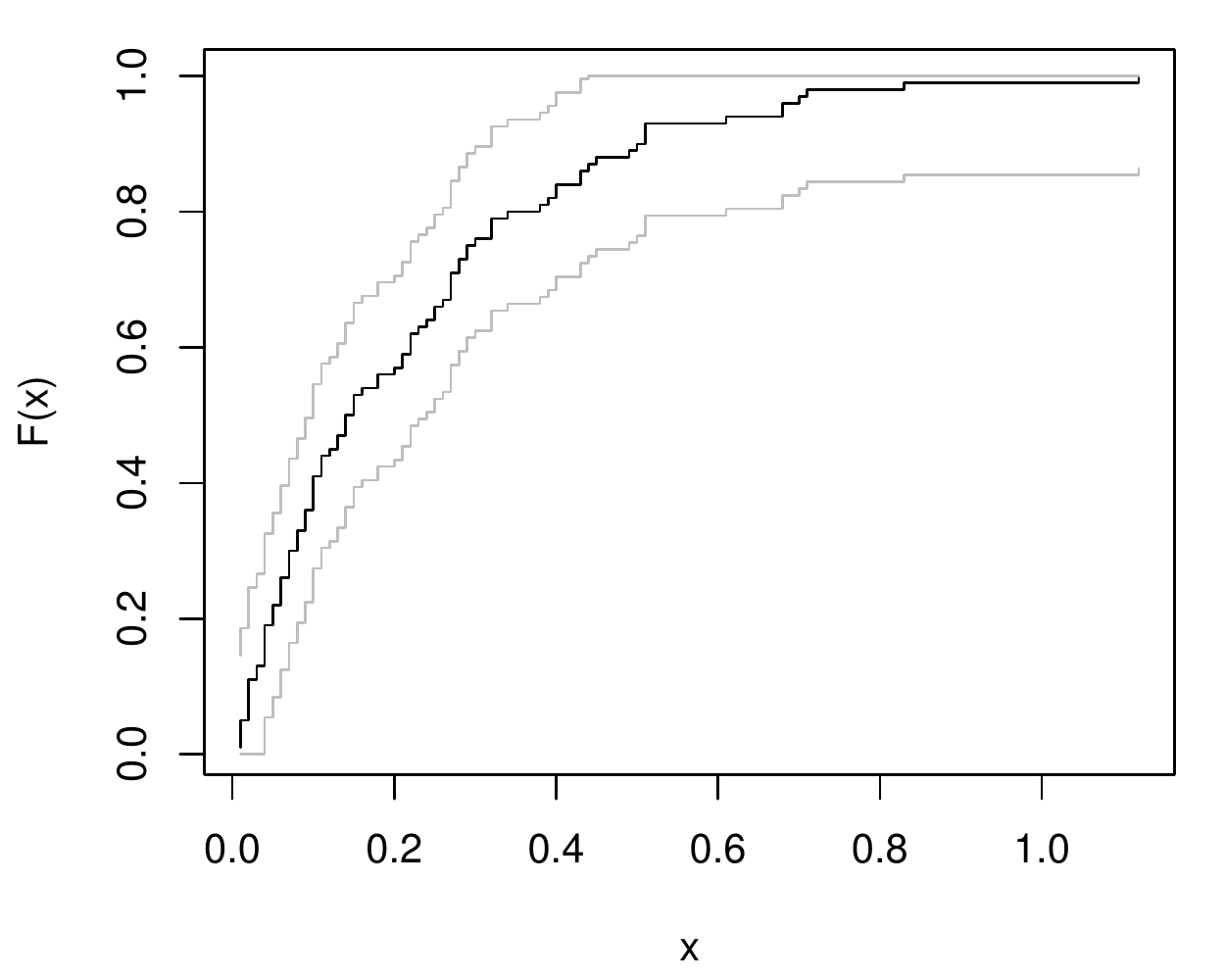}}}
\subfigure[Same as (a), Grenander estimate added]{\scalebox{0.6}{\includegraphics{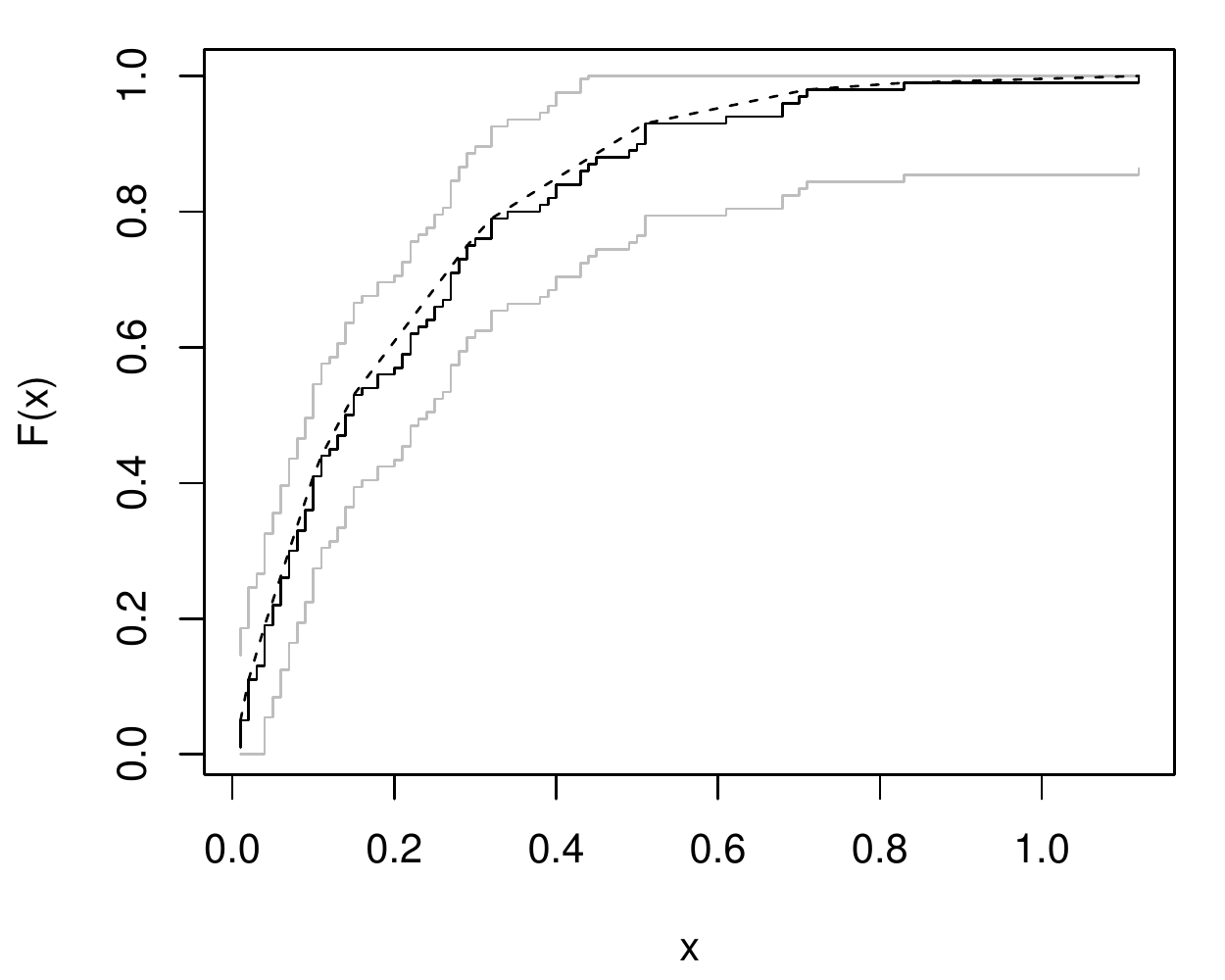}}}
\end{center}
\caption{Empirical distribution function (black) and lower and upper 95\% confidence bands (gray) for the nerve data example; subsample of size $n=100$. Panel~(b) includes the Grenander estimator $\hat F_x^\uparrow$ (dashed), which is very close to $\hat F_x$.}
\label{fig:np}
\end{figure}

A confidence band allows one to test point null hypotheses, that is, if $F_0$ is the hypothesized distribution, then simply calculate $\pi_x(F_0)$ and compare that value to the desired significance level.  But what about composite null hypotheses?  Suppose the goal is to test if the distribution function $F$ has a density $f$ that is monotone decreasing on $[0,\infty)$.  A proper test based on the DKW confidence band would be to reject if and only if the intersection between the band and the set of all distribution functions with monotone densities is empty.  However, it is a highly non-trivial computational problem to enumerate this set of distributions and compare to the given confidence band.  
I'll investigate a formal IM-based test of monotonicity in Section~\ref{SS:one.sample2}, but the formulation in terms of a plausibility contour suggests a simple approach.  That is, I could consider expressing the composite null in terms of its ``best representative'' and plug that representative in to the plausibility contour to get an overall assessment.  As a ``best representative,'' an obvious choice is the distribution function $\hat F_x^{\uparrow}$ corresponding to the maximum likelihood estimator of \citet{grenander}.  Then an informal test would reject the hypothesis that the density is monotone decreasing at level $\alpha$ if $\pi_x(\hat F_x^\uparrow)$ is less than $\alpha$.  The intuition is that if even the ``best'' of the distributions with a monotone density is implausible, then the hypothesis itself must be implausible too.  Figure~\ref{fig:np}(b) shows the Grenander estimate $\hat F_x^\uparrow$ overlaid and it's virtually indistinguishable from $\hat F_x$, which suggests that monotonicity is plausible.  By my calculation, $\pi_x(\hat F_x^\uparrow) \approx 0.85$, which confirms the suggestion from the plot.  
I make no claims here that this naive test is theoretically valid, although the superficial connections to other legitimate p-value adjustments for testing composite hypotheses \citep[e.g.,][]{berger.boos.1994} suggests it's at least a reasonable approach.

\section{Beyond the basics: an IM algorithm}
\label{S:algorithm}

To prove the existence results in Theorems~\ref{thm:complete} and \ref{thm:complete.test}, I constructed a valid IM that was no less efficient than a given confidence region or test.  The idea was to use the provided confidence region or test to define the focal elements of the (parameter-dependent) random set.  From there, I used Condition~P2 of Theorem~\ref{thm:valid} to define the distribution of the random set supported on the constructed focal elements.  And once the random set distribution is specified, the IM construction proceeds through the A-, P-, and C-steps as described in Section~\ref{SS:one.S}.  Therefore, the only step in the IM construction that must be tailored to the given problem or confidence/testing procedure is specification of the focal elements.  The final IM can, of course, be used to make valid inference about any assertion of interest and, in general, this would require some problem-specific considerations and computations which I won't discuss here. 

The IM algorithm is most effective when it can be applied to a simple and  conservative procedure.  In that case, the IM algorithm will preserve that method's validity but improve on its efficiency.  The Clopper--Pearson method was one such procedure, but that is specific to the binomial problem.  Are there any more general methods?  Recently, \citet{wasserman.universal} proposed what they call a {\em universal inference} framework which is simple and provides basic validity guarantees with virtually no assumptions.  A consequence of the method's universality is that it tends to be conservative, thus making this an ideal starting point for the aforementioned IM algorithm. For concreteness, here I will discuss the details of the IM algorithm for the case when the given procedure is a test or confidence region from the universal inference framework of \citet{wasserman.universal}.  

First I give a brief overview of their approach; they offer several variations of it, but here I'll focus on the most basic version, the one based on what they call the {\em split likelihood ratio}.  Given a statistical model $\prob_{X|\theta}$, let $\vartheta \mapsto L_X(\vartheta)$ denote the likelihood function based on data $X$ of size $n$.  Define an operators $D_1$ and $D_2$ which, when applied to a data set $x$, {\em split} the data into to disjoint subsets $D_1 x$ and $D_2 x$; for example, $D_1 x$ could be the first $\lfloor n/2 \rfloor$ components, and $D_2 x$ the remaining $n-\lfloor n/2 \rfloor$ components.  Then the basic idea is to form a likelihood ratio type statistic where one part of the data. say, $D_1 x$, is used for the likelihood function itself and then the other part, $D_2x$, is used to estimate $\theta$.  In particular, a $100(1-\alpha)$\% confidence region for $\theta$ based on the split likelihood ratio is 
\[ C_\alpha(x) = \bigl\{ \vartheta \in \Theta: L_{D_1 x}(\vartheta) \geq \alpha L_{D_1 x}(\hat\theta_{D_2 x}) \bigr\}, \]
where $\hat\theta_{D_2x}$ is, e.g., the maximum likelihood estimator of $\theta$ based on the data set $D_2 x$.  Similarly, for testing a hypothesis $H_0: \theta \in \Theta_0$, \citet{wasserman.universal} define the split likelihood ratio test as 
\[ T_\alpha(x) = 1\bigl\{ \textstyle\max_{\vartheta \in \Theta_0} L_{D_1x}(\vartheta) <  \alpha L_{D_1x}(\hat\theta_{D_2x}) \bigr\}. \]
A simple and elegant argument---without distributional assumptions or appeal to asymptotic approximations---proves that the split likelihood ratio procedures are valid in the sense that coverage probabilities for confidence regions match/exceed the advertised level and, similarly, split likelihood ratio tests control the size at or below the advertised level.  \citet{wasserman.universal} also extend these results to cover cases involving nuisance parameters as well as nonparametric problems; each of these extensions will be used below in Section~\ref{S:more}.  The drawback, however, to a method that is provably valid under no assumptions is that its performance is conservative in many applications.  This tendency to be conservative is what leaves room for the IM algorithm to improve efficiency.   

With these universal inference procedures as the starting point, the IM algorithm proceeds to construct the focal elements for the (parameter-dependent) random sets.  For an association $X=a(\theta,U)$, of course I can write $D_k X = a(\theta, D_k U)$, $k=1,2$.  Then the focal elements for the split likelihood ratio confidence region is 
\begin{align*}
S_\alpha(\vartheta) & = \{ u: C_\alpha(a(\vartheta,u)) \ni \vartheta \} \\
& = \bigl\{u: L_{a(\vartheta, D_1 u)}(\vartheta) \geq \alpha L_{a(\vartheta, D_1 u)}(\hat\theta_{a(\vartheta, D_2 u)}) \bigr\}, \quad \vartheta \in \Theta. 
\end{align*}
Similarly, for the split likelihood ratio test, the focal elements are 
\begin{align*}
S_\alpha(\vartheta) & = \{u: T_\alpha(a(\vartheta, U)) = 0\} \\
& = \bigl\{ u: \max_{t \in \Theta_0} L_{a(\vartheta, D_1 u)}(t) \geq \alpha L_{a(\vartheta, D_1 u)}(\hat\theta_{a(\vartheta, D_2u)}) \bigr\}, \quad \vartheta \in \Theta_0.
\end{align*}
Just like the examples in Section~\ref{S:examples}, the random set distributions, i.e., $\prob_{\S|\vartheta}$, are determined by two things: first, the distribution $\prob_U$ of $U$, which is defined by the user's choice of association $X=a(\theta,U)$ and, second, the $\prob_U$-probabilities of these focal elements, as described in Condition~P2 of Theorem~\ref{thm:valid}.  

Since the split likelihood ratio framework is more abstract than those in the previous examples above, I'll give a quick simple illustration here. Suppose that $X=(X_1,\ldots,X_n)$ is an iid sample from $\nm(\theta,1)$ and the goal is to construct a confidence interval for $\theta$.  Assume that $n=2m$ and that the data splitting is into the first and second groups, each of size $m$.  First, for the association, I make the obvious choice $X= a(\theta, U)$, where $a(\vartheta,u) = \vartheta 1_n + u$ and $U \sim \prob_U := \nm_n(0,I)$.  Then the focal elements on which the random set will be supported are given by 
\begin{align*}
S_\alpha(\vartheta) & = \{u \in \RR^n: C_\alpha(a(\vartheta, u)) \ni \vartheta\} \\
& = \{(u_1,u_2) \in \RR^m \times \RR^m: L_{a(\vartheta, u_1)}(\vartheta) \geq \alpha L_{a(\vartheta, u_1)}(\hat\theta_{a(\vartheta, u_2)}) \} \\
& = \{(u_1,u_2) \in \RR^m \times \RR^m: -\bar u_2^2 + 2 \bar u_1 \bar u_2 \leq -\tfrac{4}{n} \log\alpha\}, 
\end{align*}
where $\bar u_k = m^{-1} \sum_{j=1}^m u_{kj}$ is the mean of the $m$-vector $u_k$, for $k=1,2$.  Note that the focal elements do not depend on $\vartheta$, i.e., $S_\alpha(\vartheta) \equiv S_\alpha$, which is not surprising given the location parameter form of the model.  Second, it is easy to verify that 
\[ \alpha(x,\vartheta) = \min\bigl\{1 \, , \,  L_{D_1x}(\vartheta) /  L_{D_1x}(\hat\theta_2) \bigl\}, \]
with $\hat\theta_2$ the average of the $m$-vector $D_2x$, and the upper bound at 1 will in some cases, like in the binomial problem above, result in a plateau on the plausibility contour; however, for large enough $n$, the probability of seeing a plateau will be rather small.  Then that plausibility contour is given by 
\[ \pif_x(\vartheta) = 1-\prob_U(S_{\alpha(x,\vartheta)}) = \prob_U\bigl\{ -\bar U_2^2 + 2 \bar U_1 \bar U_2 > -\tfrac{4}{n} \log \alpha(x,\vartheta) \bigr\}, \]
which is easy to evaluate via Monte Carlo. This could easily be optimized as well to evaluate the plausibility assigned to any assertion about $\theta$. 

To see what the result of this construction looks like, I generated a sample of size $n=2m=200$ from a standard normal distribution.  Figure~\ref{fig:slr.normal} shows a plot of the p-value function from the split likelihood ratio test (whose level sets correspond to the split likelihood ratio confidence interval), the IM-based plausibility contour, and, for reference, the ``optimal'' plausibility contour based on the known sampling distribution of the sample mean.  Both the split likelihood ratio and the IM-based improvement have the plateau in a neighborhood around the full sample mean, but otherwise the latter has much narrower contours compared to the former, hence is considerably more efficient.  As expected, both of these procedures are far less efficient than the optimal procedure.  However, the real value of the split likelihood ratio approach, and the IM algorithm, is for problems where no ``optimal'' procedure is available, like in the two examples below.

\begin{figure}[t]
\begin{center}
\scalebox{0.70}{\includegraphics{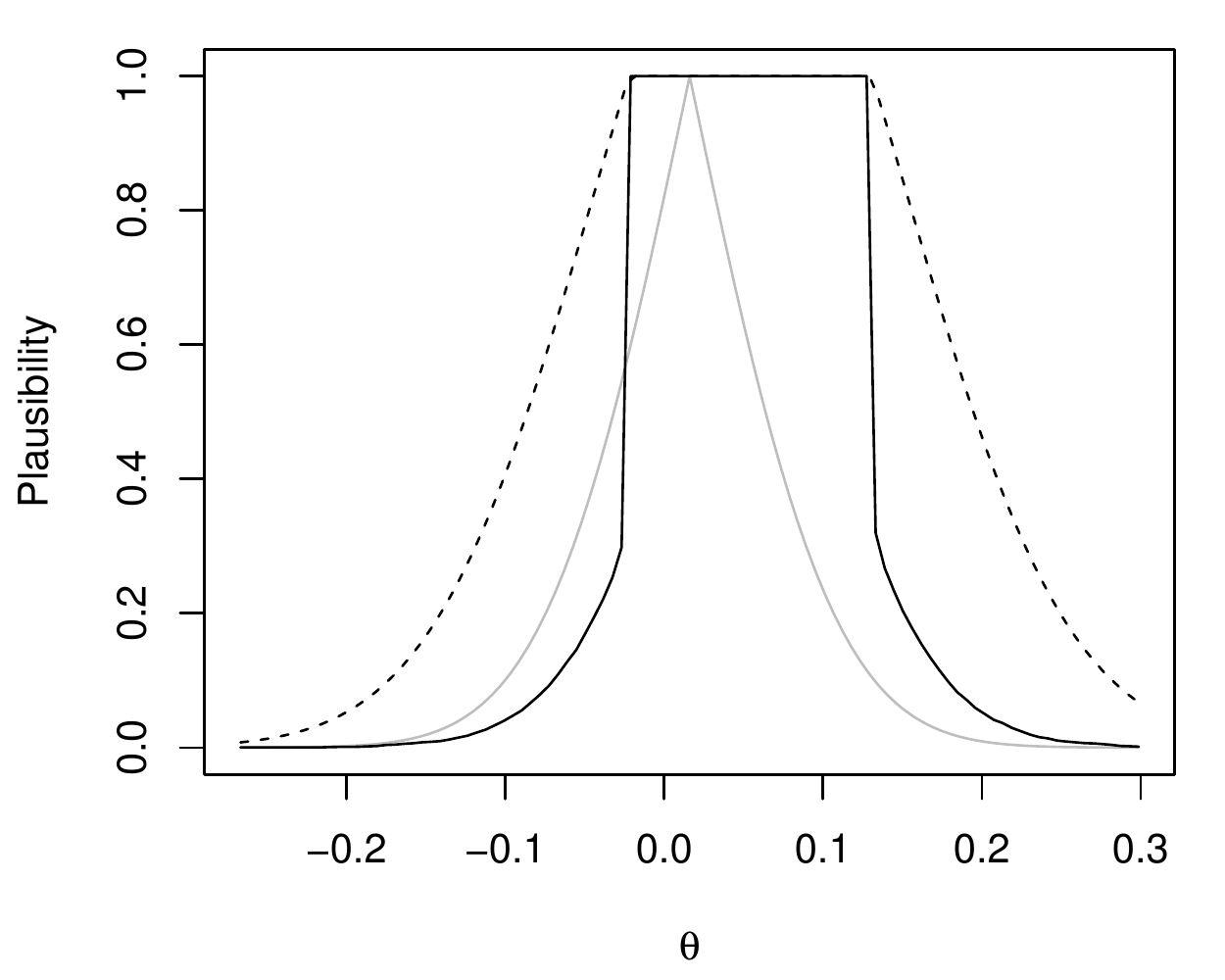}}
\end{center}
\caption{Comparison of plausibility contours based on the universal split likelihood ratio procedure (dashed), the improvement based on the IM algorithm (solid), and the optimal one based on the standard $z$-test (gray).}
\label{fig:slr.normal}
\end{figure}

\section{Two more illustrations}
\label{S:more}

\subsection{Mixture model testing problem}

An important class of statistical problems is those that involve mixture distributions.  The mixture model formulation makes it possible to capture rather complex distributional forms with a relatively small number of individually simple components.  This is what makes mixture models ideal for density estimation applications.  However, despite the simplicity and ubiquity, mixture models themselves are notoriously difficult, in large part due to singularities \citep[e.g.,][]{drton2009} that makes the standard asymptotic distribution approximation of \citet{wilks1938} inapplicable.  So it's remarkable that, with so many statisticians having investigated the problem of inference in mixture models over the years, there was, until just this year, no test of one normal distribution versus a mixture of two normal distributions that provably controls the Type~I error rates at the advertised level, not even asymptotically.  Of course, various tests are available that show strong empirical performance \citep[e.g.,][]{mclachlan1987}, but none were proved to be theoretically valid until the very recent seminal work of \citet{wasserman.universal}.   

As described in Section~\ref{S:algorithm} above, as soon as a provably valid test or confidence region is available, it can be used as input into the IM algorithm to improve its efficiency.  To illustrate this, I'll reproduce the example presented in Section~4 of \citet{wasserman.universal} to compare both the size and power of their proposed split likelihood ratio test of one versus two normal components in a mixture model.  More specifically, I consider the class of true mixture distributions 
\begin{equation}
\label{eq:power.mix}
\tfrac12 \nm(-\mu, 1) + \tfrac12 \nm(\mu, 1), \quad \text{indexed by $\mu \geq 0$}, 
\end{equation}
with $\mu=0$ being the null.  Figure~\ref{fig:power.mix} shows the power, as a function of $\mu \geq 0$, for both the split likelihood ratio and corresponding IM-based test, with $n=100$, using level $\alpha=0.05$.  It is clear that both tests control the Type~I error at 0.05, confirming the theory, but the IM test has considerably higher power.  The general theory presented above suggests that this gain in efficiency is not specific to this simple one- versus two-component testing problem, but a detailed investigation will be carried out elsewhere.

\begin{figure}[t]
\begin{center}
\scalebox{0.65}{\includegraphics{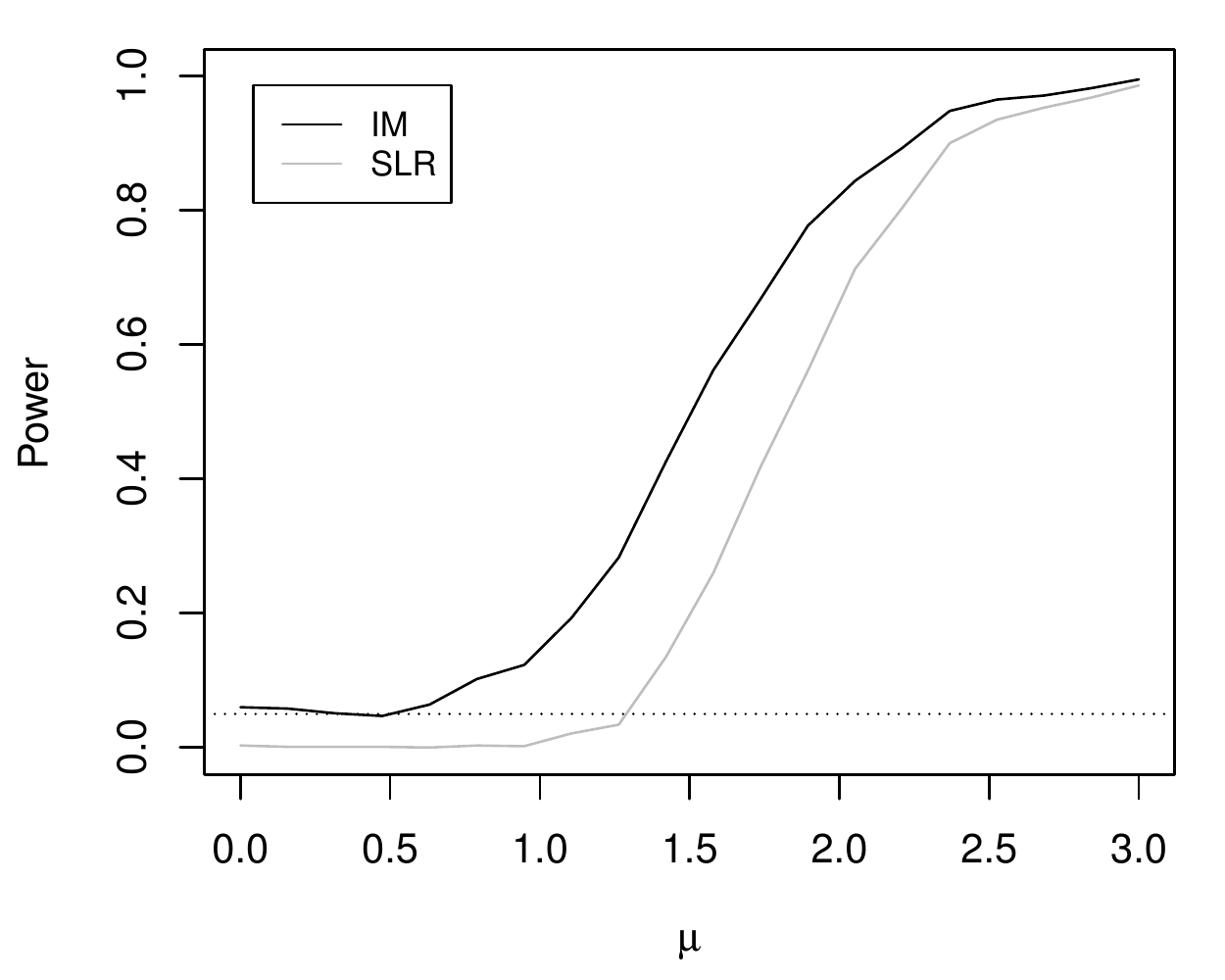}}
\end{center}
\caption{Mixture model problem, testing one component versus two.  Plot shows the power of the split likelihood ratio test and that of the IM test as functions of $\mu$, with $n=100$ and $\alpha=0.05$.}
\label{fig:power.mix}
\end{figure}

\subsection{Nonparametric one-sample problem, revisited}
\label{SS:one.sample2}

Reconsider the one-sample problem from Section~\ref{SS:one.sample} above.  Here, however, I'll assume that the distribution $F$ has a density, $f$, supported on $[0,\infty)$, and that the goal is to test if $f$ is monotone non-increasing, i.e., $H_0: f \in \FF_0$, where $\FF_0$ denotes the class of decreasing densities on $[0,\infty)$.  Estimation of a monotone density is a classical problem, and the nonparametric maximum likelihood estimator was developed by \citet{grenander} and extensively studied by \citet{prakasarao}, \citet{groeneboom}, and \citet{bala.wellner.2007}, among others.  In the context of regression, testing for monotonicity and inference under shape constraints has been more widely studied, e.g., \citet{bowman.etal.1998}, \citet{banerjee.wellner.2001}, \citet{cai.shape.2013}, and \citet{chakraborty.ghosal.proj, chakraborty.ghosal.testreg}.  But the literature on testing a density for monotonicity is rather scarce---to my knowledge, the only detailed investigation is in a very recent technical report, \citet{chakraborty.ghosal.testdensity}.  However, this problem can be directly tackled using the general framework developed in \citet{wasserman.universal}, and below I describe one such test and the corresponding IM-based improvement.  

For an iid sample $X=(X_1,\ldots,X_n)$ from a distribution with density $f$, the likelihood function is defined as 
\[ f \mapsto L_X(f) = \prod_{i=1}^n f(X_i), \quad f \in \FF_0. \]
If the density is assumed to be monotone decreasing on $[0,\infty)$, then the nonparametric maximum likelihood estimator, $f_X^\dagger$, exists, is a piecewise constant function, and equal to the left-continuous derivative of the least concave majorant of the empirical distribution function \citep{grenander}; I use the {\tt grenander} function in the R package {\tt fdrtool} \citep{fdrtool} to evaluate $f_X^\dagger$ in my illustration below.  I also need an estimator of $f$ without the monotonicity constraint, so I define $\hat f_X$ to be a kernel density estimator---based on the default settings in the {\tt density} function in R---fit to log-transformed data $X$ and then transformed back to the interval $[0,\infty)$.  Using a pair of data-splitting operators $D_1$ and $D_2$ as described above, the universal split likelihood ratio test I employ here is given by 
\[ T_\alpha(x) = 1\bigl\{ L_{D_1x}(f_{D_1x}^\dagger) < \alpha L_{D_1x}(\hat f_{D_2x}) \bigr\}. \]
Note that I make no claims that my suggested approach for dealing with the likelihood under the alternative is ``optimal'' in any way.  Regardless, the general theory in \citet{wasserman.universal} implies that the above test provably controls the Type~I error rate at level $\alpha$, though it could be very conservative.  To my knowledge, this is the only test available in the literature that can make this claim.  

For the IM-based improvement, take the association $X=a(f,U)$ to be the same as in Section~\ref{SS:one.sample}, i.e., 
\[ X_i = F^{-1}(U_i), \quad i=1,\ldots,n, \]
where $F$ is the distribution function corresponding to the density $f$ in $\FF_0$ and $U=(U_1,\ldots,U_n)$ is a $n$-vector of iid $\unif(0,1)$ random variables.  Then, for a generic monotone density $f$, the focal elements $S_\alpha(f)$ are given by 
\[ S_\alpha(f) = \bigl\{u: L_{a(f,D_1u)}(f_{a(f,D_1u)}^\dagger) \geq \alpha L_{a(f,D_1u)}(\hat f_{a(f,D_2u)}) \}. \]
Unlike the previous example, these focal elements depend on the generic $f$.  This dependence is not a problem for evaluating the plausibility contour at a given $f$, 
\[ \pif_x(f) = 1 - \prob_U\{S_{\alpha(x,f)}(f)\}, \]
where 
\[ \alpha(x,f) = \min\bigl\{1 \, , \, L_{D_1x}(f) / L_{D_1x}(\hat f_{D_2x}) \bigr\}. \]
The challenge, of course, is that the corresponding test requires maximizing $\pif_x(f)$ over $\FF_0$.  The ``most plausible'' monotone density at which this maximum is attained should be close to the ``most likely'' monotone density, which is the Grenander estimator evaluated based on the entire data set $x$.  Therefore, I propose the approximation 
\begin{equation}
\label{eq:mono.approx}
\uPif_x(\FF_0) := \sup_{f \in \FF_0} \pif_x(f) \approx \pif_x(f_x^\dagger). 
\end{equation}
Since the right-hand side is generally smaller than the supremum on the left-hand side, a test based on this approximation is generally more aggressive than the genuine IM-based test investigated in Theorem~\ref{thm:complete.test}.  However, the split likelihood ratio test used to construct it is rather conservative, so I have not noticed any loss of validity based on the suggested more-aggressive approximation.  A detailed investigation into the test's performance will be presented elsewhere.  Of course, there is nothing special about testing monotonicity---a similar split likelihood ratio approach can be followed to test for other shape constraints, such as log-concavity \citep{bala.etal.2009}.  

For a quick illustration, consider a gamma distribution with shape parameter $\xi > 0$ and scale parameter fixed at 1.  The case $\xi=1$ corresponds to a monotone density, where the null hypothesis is true.  For $\xi > 1$, the density is not monotone, and I'll investigate here the power of the split likelihood ratio and corresponding IM-based tests as a function of $\xi$.  Here I take $n=300$ and, as before, split the data into its first and second halves.  Figure~\ref{fig:power.mono} displays the power function for the two tests, with specified significance level $\alpha=0.05$.  At the null hypothesis $\xi=1$, both tests have power/size much lower than the specified level, suggesting these are relatively conservative.  However, as $\xi$ increases, the rate at which the power of the IM-based test increases is considerably faster than that of the split likelihood ratio test.  This strongly suggests that the former is both valid and more powerful than the latter, consistent with the conclusion in Theorem~\ref{thm:complete.test}.  It remains to better understand the effect of the admittedly-naive approximation in \eqref{eq:mono.approx}.  

\begin{figure}[t]
\begin{center}
\scalebox{0.65}{\includegraphics{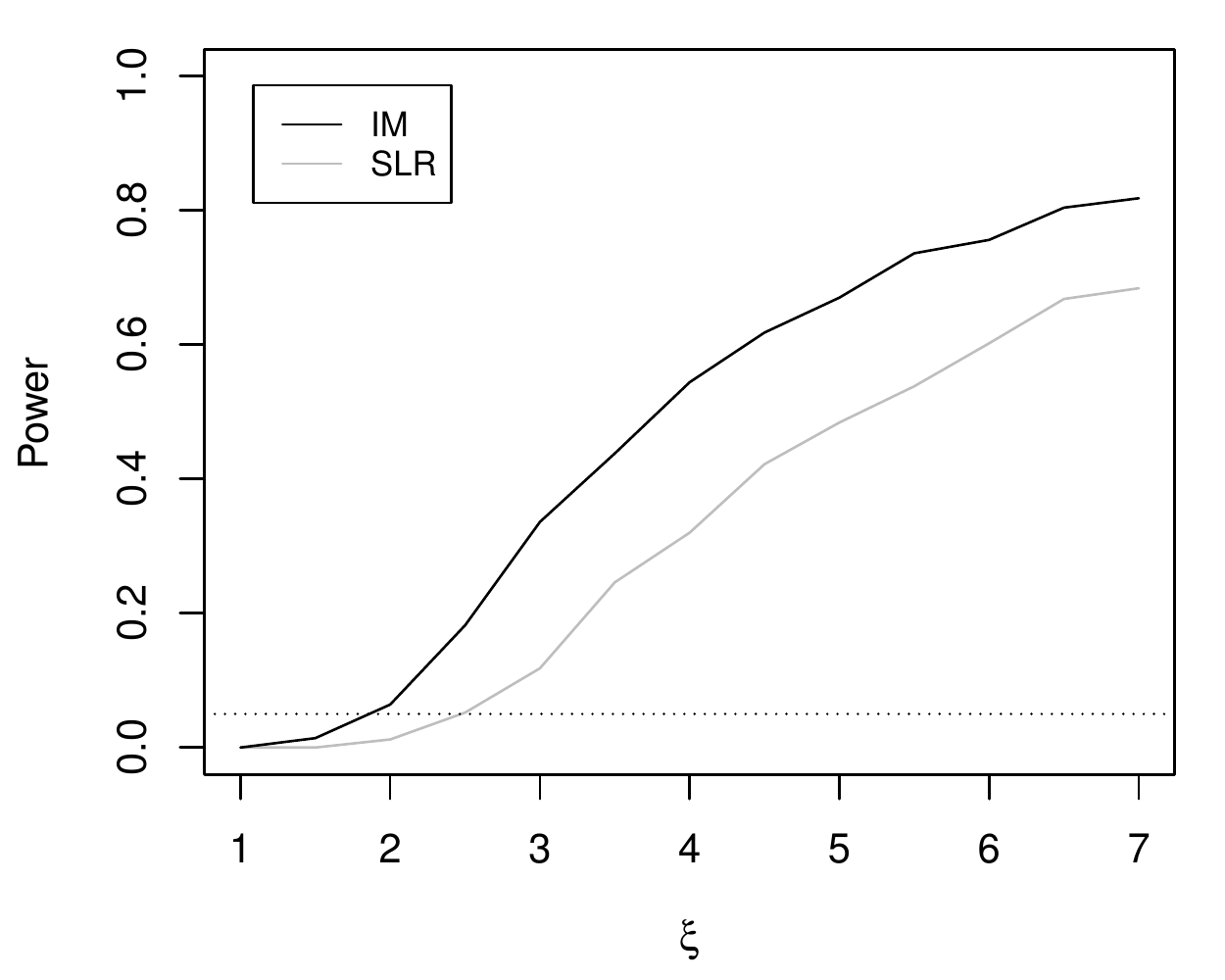}}
\end{center}
\caption{Testing monotonicity of the density in a nonparametric one-sample problem.  Plot shows the power of the split likelihood ratio test and that of the IM test as functions of the gamma shape parameter $\xi$, with $n=300$ and $\alpha=0.05$.}
\label{fig:power.mono}
\end{figure}

Finally, revisiting the nerve data example from Section~\ref{SS:one.sample}, with the subsample of size $n=100$, the split likelihood ratio statistic is about 1.7, so the universal test $T_\alpha(x)$ as defined above does not reject at any $\alpha$ level.  This also implies that $\alpha(x,f_x^\dagger) = 1$ and so the IM algorithm returns plausibility $\uPif_x(\FF_0) = 1$.  This conclusion agrees with what's shown in Figure~\ref{fig:np}, where the empirical distribution function almost exactly agrees with the best monotone approximation.

\section{Extensions and open problems}
\label{S:open}

\subsection{Restricting the class of assertions}
\label{SS:rich}

The validity condition in Definition~\ref{def:valid} puts constraints on the capacity $\uPi_x$ and/or on the collection $\A$ of assertions.  My approach here was to keep the collection $\A$ as rich as possible, e.g., the entire power set, and see what requirements are needed on the capacity to achieve validity.  It turns out that consonance is an important (maybe necessary---see below) condition for a capacity to achieve validity for such a rich collection of assertions.  A natural question is if more flexibility can be achieved in the capacity's structure by restricting the collection $\A$ of assertions.  

Consider a simple case with a scalar $\theta$.  As I argued in Section~\ref{SS:isnt.valid}, in order for the capacity to be a probability measure, the class $\A$ would need to be restricted to half-lines in order to maintain validity.  This is unsatisfactory to me because it means, e.g., the effort to convert confidence intervals into a confidence distribution was for nothing, since it can only reliably be used to read off confidence intervals.  

But there would be less extreme cases that deserve consideration.  Speaking in terms of the dual $\lPi_x$ of $\uPi_x$, since the terminology is more familiar and standard, an interesting question is the following: for what collection $\A$ of assertions might a $k$-monotone capacity \citep[e.g.,][]{huber1973.capacity, huber.strassen.1973, sundberg.wagner.1992}, a belief function or $\infty$-monotone capacity \citep[e.g.,][]{shafer1976, shafer1979, nguyen1978}, or a more general lower prevision \citep[e.g.,][]{walley1991, miranda2008, lower.previsions.book} be valid?  

Another seemingly innocuous restriction is the requirement that \eqref{eq:valid} hold for all $\alpha \in [0,1]$.  This implies that $(x,A) \mapsto \uPi_x(A)$ can and will take values arbitrarily close to 1.  It seems to me that this more or less implies that $\uPi_x$ must be a consonant plausibility function.  More flexibility can be achieved if validity was relaxed so that \eqref{eq:valid} held for all $\alpha \in [0,\bar\alpha]$, with $\bar\alpha < 1$.  \citet{walley2002} showed that a carefully constructed generalized Bayes solution could achieve this kind of relaxed validity condition.  However, his ``$\bar\alpha$'' is directly tied to the size of his prior credal set, so achieving even this relaxed validity property from a generalized Bayes approach loses a lot of efficiency compared to solutions I'm advocating for based on consonant plausibility measures.  The generalized Bayes solution has other nice properties that the plausibility measures don't, e.g., the strong coherence property in \citet[][Theorem~7.8.1]{walley1991}, so it remains to better understand the balance between the various properties.






\subsection{Fused plausibility}
\label{SS:open.fused}

Section~\ref{SS:family.S} introduced the notion of a fused plausibility measure, but only when the condition \eqref{eq:max} holds, can the resulting capacity $\uPi_x$ really be called a plausibility measure with the consonance property.  Interestingly, \citet[][Theorem~4.4]{miranda.etal.2004} shows that, in all practical cases, plausibility measures are equivalent to the distribution of a random set.  This means that, in addition to the representation in terms of a family of parameter-dependent random set distributions on the $\UU$-space, there is another representation of the output $\uPi_x$ in terms of a {\em data-dependent} random set on $\Theta$.  This latter random set can be pulled back, through the association, to a random set on $\UU$, but it would generally still depend on data.  This suggests a potentially interesting duality between the parameter-dependent random sets and the fusion operation in Section~\ref{SS:family.S} and the data-dependent random sets that have appeared in, e.g., \citet{leafliu2012} and \citet{imexpert}.  The question is: how can this duality be leveraged for improved understanding and for use in new and challenging problems?



\subsection{Approximate validity}

My focus here has been on situations where a test or confidence region is available that exactly achieves the error rate control properties in \eqref{eq:size} or \eqref{eq:coverage}, respectively.  Thanks to the recent developments in \citet{wasserman.universal} discussed in Section~\ref{S:algorithm}, my starting point covers virtually every problem.  However, it's still worth asking what happens if the procedure used as input the IM algorithm was only ``approximately valid'' in some specific sense, e.g., as the sample size $n$ approaches infinity. For example, will using an only approximately valid input to the IM algorithm return a procedure that's still approximately valid and potentially more efficient?

\subsection{Beyond inference}
\label{SS:decision}

My focus in this paper has been exclusively on statistical inference, but this is not the only problem that statisticians and data scientists are interested in.  One of those important problems is {\em prediction}, the use of observed data and any other relevant information to predict the next observation.  A first IM approach to prediction,  based on a well-specified statistical model with parameter $\theta$, along the lines presented in Section~\ref{SS:one.S}, was given by \citet{impred}.  More recently, \citet{imconformal, impred.isipta} showed that conformal prediction \citep[e.g.,][]{vovk.shafer.book1, shafer.vovk.2008} is naturally understood as an imprecise probability and, moreover, that it can be similarly characterized using random sets as I did in this paper for p-values and confidence regions.  These results showcase the flexibility and potential of IMs and, furthermore, they help to position conformal prediction as part of an (imprecise) probabilistic inference framework.  

Another relevant problem is decision-making, and there is a close connection between probability and formal decision theory.  Indeed, it's very natural to encode ones preferences in terms of a utility or loss function and define the optimal decision as that which suitably optimizes the expected utility or loss.  Since a framework that deals with probabilities has an immediate connection to decision theory, this is often seen as a win for the Bayesian side.  However, having a formal connection between the frequentist framework and imprecise probabilities changes this perspective, because there is a well-developed theory of optimal decisions in the imprecise probability literature; see, e.g., \citet{denoeux.decision.2019} for a recent survey.  So, it would be interesting to investigate the valid IMs developed here in the context of a formal imprecise-probabilistic decision theory framework.  

A related matter is coherence in the sense of \citet{definetti1937}.  Of course, de Finetti's theory suggests that coherence can only be achieved via (finitely additive) probabilities, but there is now an extensive literature \citep[e.g.,][]{walley1991, williams.previsions} showing that de Finetti's theory is a special case of a more general notion of coherence that holds for imprecise probabilities.  So if the frequentist output can be understood as an imprecise probability, as I showed here, then there ought to be a parallel frequentist theory of coherence, one that balances the statistical considerations of validity as in Definition~\ref{def:valid} with those logical considerations of avoiding sure loss, etc.  Some first developments on this are given in \citet{imconformal} but in the context of prediction; see, also, \citet{gong.meng.update} and the related discussion in \citet{gong.meng.discuss}.  New results along these lines are in development and will be reported elsewhere.

\section{Conclusion}
\label{S:discuss}

Motivated by a desire to understand what is required to achieve valid probabilistic inference, and a recognition that ordinary additive probabilities are limited unless genuine prior information is available, this paper had the ambitious goal to characterize the so-called ``frequentist approach'' via imprecise probabilities.  To achieve this goal, I showed that for every suitable frequentist procedure with error rate control guarantees, there exists a valid plausibility measure such that the corresponding procedure derived from it is at least as efficient as the given procedure.  Moreover, these valid plausibility measures have the additional structure of those derived from the framework of \citet{imbook}.  This puts the Bayesian and frequentist schools on a common ground in terms of having a mathematically rigorous underpinning, namely, precise and imprecise probability theory, respectively.  As discussed in Section~\ref{S:intro}, this common ground provides an opportunity to clearly explain the interpretation of p-values and confidence, which would strengthen our reputation in the eyes of other scientists and improve the quality of scientific communication.  A further consequence of this common ground is that, with a slight generalization of the validity condition \eqref{eq:valid} taken here, the nature, form, and strength of prior information determine where on the spectrum between precise and imprecise probability one needs to be for validity.  These details will be reported elsewhere.



I argued in Section~\ref{S:intro} that this imprecise probabilistic perspective can have a positive impact on education and scientific communication.  An additional benefit is that it wouldn't require a complete education reform.  For advanced courses, the instructor may opt to present details about imprecise probability, but this isn't necessary.  Theorems~\ref{thm:complete}--\ref{thm:complete.test} show that all the p-values and confidence intervals presented in introductory courses (e.g., based on Student's t-test, analysis of variance, and linear regression) correspond to plausibility measures, so the instructor can simply give students the intuitive plausibility-based interpretations as I explained in Section~\ref{S:intro}.  This is precisely what I do in the courses I teach.  Based on the feedback I've received, students are able to understand my interpretation of p-values and confidence even if they don't know anything about the mathematics of imprecise probability.\footnote{For example, a masters student who took my course wrote this to me after her job interview: ``I wanted to let you know that your discussion of plausibility in ST503 made it easier for me to explain what a p-value is when I was interviewing.  Thank you for the new perspective on the subject!''}  

At the {\em BFF4} meeting at Harvard in 2017,\footnote{BFF stands for both ``Bayes, fiducial, and frequentist'' and ``best friends forever.''  See  \citet{xl.bff.2017} for details about the group and its mission, and for more explanation about the name.} Nancy Reid's invited talk summarized each of the different approaches discussed at the meeting, with a one-liner from the respective advocates of those approaches.  Professor Reid's one-liner for me was ``IMs are the only answer.'' It took me a few years to finish this paper and justify that claim, but better late than never.  The point is that everyone wants procedures with good properties and, as the results of Section~\ref{S:main} show, whatever approach one takes to find a good procedure, it's ultimately a valid IM based on the developments in \citet{imbook}.

\section*{Acknowledgments}

The author is grateful to Duncan Ermini Leaf and Chuanhai Liu for valuable feedback on a previous version of this manuscript, and to Michael Balch and Leonardo Cella for discussions that helped to shape the arguments presented in this revised version.  This work is partially supported by the U.S.~National Science Foundation, under grants DMS--1811802 and SES--2051225.


\bibliographystyle{apalike}
\bibliography{/Users/rgmarti3/Dropbox/Research/mybib}

\end{document}